\documentclass[12pt]{article}

\usepackage[british]{babel}

\usepackage[T1]{fontenc}

\usepackage{mathptmx}
\usepackage[scaled=.90]{helvet}
\usepackage{courier}

\usepackage[pdftex,breaklinks=true,colorlinks=true,pdfstartview=FitH,       urlcolor=blue]{hyperref}

\usepackage{amsmath}
\DeclareMathOperator{\tr}{tr}
\DeclareMathOperator{\RE}{Re}

\usepackage{amsfonts}
\usepackage{amsbsy}
\usepackage{amssymb}

\usepackage{amsthm}

\theoremstyle{plain}
\newtheorem{theorem}{Theorem}

\theoremstyle{definition}
\newtheorem{definition}{Definition}

\begin{document}

\title{Invariant random fields in vector bundles and application to cosmology\thanks{This work is supported by the Swedish Institute grant SI--01424/2007.}}

\author{Anatoliy Malyarenko\thanks{Division of Applied Mathematics, School of Education, Culture and Communication,
Mälardalen University, SE 721 23 Västerås, Sweden. E-mail: \texttt{anatoliy.malyarenko@mdh.se}}}

\date{\today}

\maketitle

\begin{abstract}
We develop the theory of invariant random fields in vector bundles. The spectral decomposition of an invariant random field in a homogeneous vector bundle generated by an induced representation of a compact connected Lie group $G$ is obtained. We discuss an application to the theory of cosmic microwave background, where $G=SO(3)$. A theorem about equivalence of two different groups of assumptions in cosmological theories is proved.
\end{abstract}

\section{Introduction}

This paper is inspired by Geller and Marinucci (2008). After reading the above paper and several physical books and papers cited below, the author realised that cosmological applications require the theory of random fields in vector bundles. A variant of such a theory is developed in Section~\ref{sec:fields}, while an application to cosmology is described in Section~\ref{sec:cosmology}.

According to vast majority of modern cosmological theories, our Universe started in a ``Big Bang". This term refers to the idea that the Universe has expanded from a hot and dense initial condition at some \emph{finite} time in the past, and continues to expand now.

As the Universe expanded, both the plasma and the radiation grew cooler. When the Universe cooled enough, it became transparent. The photons that were around at that time are observable now as the \emph{relic radiation}. Their glow is strongest in the microwave region of the radio spectrum, hence another name \emph{cosmic microwave background radiation}, or just CMB.

In cosmological models, it is usually assumed that the CMB is a single realisation of a random field. A CMB detector measures an electric field $\mathbf{E}$ perpendicular to the direction of observation (or line of sight) $\mathbf{n}$. Mathematically, $\mathbf{n}$ is a point on the sphere $S^2$. The vector $\mathbf{E}(\mathbf{n})$ lies in the tangent plane, $T_{\mathbf{n}}S^2$. In other words, $\mathbf{E}(\mathbf{n})$ is a section of the tangent bundle $\xi=(TS^2,\pi,S^2)$ with
\[
\pi(\mathbf{n},\mathbf{x})=\mathbf{n},\qquad\mathbf{n}\in S^2,\quad\mathbf{x}\in T_{\mathbf{n}}S^2.
\]

It follows that cosmology uses the theory of random fields in vector bundles. A short introduction to vector bundles may be found in Geller et al (2009). It is not difficult to give a formal definition of a random field in a vector bundle. Indeed, let $\mathbb{K}$ be either the field of real numbers $\mathbb{R}$ or the field of complex numbers $\mathbb{C}$. Let $\xi=(\mathcal{E},\pi,T)$ be a finite-dimensional $\mathbb{K}$-vector bundle over a Hausdorff topological space $T$.

\begin{definition}
A \emph{vector random field} on $\xi$ is a collection of random vectors $\{\,\mathbf{X}(t)\colon t\in T\,\}$ satisfying $\mathbf{X}(t)\in\pi^{-1}(t)$, $t\in T$.
\end{definition}

In other words, a vector random field on the base $T$ of the vector bundle $\xi$ is a random section of $\xi$.

To define a second order vector random field, assume that every space $\pi^{-1}(t)$ carries an inner product.

\begin{definition}
A vector random field $\mathbf{X}(t)$ is \emph{second order} if $\mathsf{E}\|\mathbf{X}(t)\|^2_{\pi^{-1}(t)}<\infty$, $t\in T$.
\end{definition}

Next, we try do define a mean square continuous random field. The naive approach
\[
\lim_{s\to t}\mathsf{E}\|\mathbf{X}(s)-\mathbf{X}(t)\|^2=0
\]
does not work. If $s$, $t\in T$ with $s\neq t$, then $\mathbf{X}(s)$ and $\mathbf{X}(t)$ lie in different spaces. Therefore, the expression $\mathbf{X}(s)-\mathbf{X}(t)$ is not defined.

To overcome this difficulty, we extend an idea of Kolmogorov formulated by him for the case of a trivial vector bundle and  published by Rozanov (1958) and Yaglom (1961). We start Subsection~\ref{sub:definitions} by defining a scalar random field on the total space $\mathcal{E}$, which we call the field \emph{associated} to the vector random field $\mathbf{X}(t)$. Then, we call $\mathbf{X}(t)$ mean square continuous if the associated scalar random field is mean square continuous.

Let $G$ be a topological group acting continuously from the left on the base $T$. We would like to call a vector random field $\mathbf{X}(t)$ \emph{wide sense left $G$-invariant}, if the associated scalar random field is wide sense left $G$-invariant with respect to some left continuous action of $G$ on the total space $\mathcal{E}$. However, in general there exist no natural continuous left action of $G$ on $\mathcal{E}$. In Definition~\ref{def:associated}, we define an action of $G$ on $\mathcal{E}$ \emph{associated} to its action on the base space $T$. Then, we call a vector random field $\mathbf{X}(t)$ wide sense left $G$-invariant, if the associated scalar random field is wide sense left $G$-invariant with respect to the associated action.

In Subsection~\ref{sub:associated}, we consider an important example of an associated action: the so called \emph{homogeneous, or equivariant} vector bundles. They are important for us by several reasons.

On the one hand, they have a natural associated action of some topological group~$G$. Moreover, the above action identifies the vector space fibers over any two points of the base space. Therefore, all random vectors of a random field $\mathbf{X}(t)$ in a homogeneous vector bundle lie in the same space. We prove that for homogeneous vector bundles, our definitions of mean square continuous field and invariant field are equivalent to usual definitions \eqref{eq:common} and \eqref{eq:invariant}.

On the other hand, the space of the square integrable sections of a homogeneous vector bundle carries the so called \emph{induced representation} of the group~$G$. Therefore, we can use the well-developed theory of induced representations to obtain spectral decompositions of invariant random fields in homogeneous vector bundles. For an introduction to induced representations, see Barut and R\k{a}czka (1986).

In Subsection~\ref{sub:spectral} we consider mean square continuous random fields in homogeneous vector bundles over a homogeneous space $T=G/K$ of a compact connected Lie group $G$. In Theorem~\ref{th:1}, we prove the spectral decomposition of a random field in a homogeneous vector bundle of the representation of the group~$G$ induced by an irreducible representation of its subgroup~$K$. Here, we first meet the system of functions ${}_W\mathbf{Y}_{Vm}(t)$ defined by \eqref{eq:spherical}, which form the orthonormal basis in the space of the square integrable sections of a homogeneous vector bundle under consideration. The spectral decomposition in Theorems~\ref{th:1}--\ref{th:3} is given in terms of the above functions.

In Theorem~\ref{th:2}, we find the restrictions under which the spectral decomposition of Theorem~\ref{th:1} describes a wide sense $G$-invariant random field. Finally, Theorem~\ref{th:3} is a generalisation of Theorem~\ref{th:2} to the case when the representation of the group~$G$ is induced by a direct sum of finitely many irreducible representations of the subgroup~$K$.

In Section~\ref{sec:cosmology}, we apply theoretical considerations of Section~\ref{sec:fields} to cosmological models. Subsection~\ref{sub:deterministic} is a short introduction to the deterministic model of the CMB for mathematicians. In particular, we discuss different choices of local coordinates in the tangent bundle $\xi=(TS^2,\pi,S^2)$, and fix our choice. We explain both the mathematical and physical sense of the \emph{Stokes parameters} $I$, $Q$, $U$, and $V$. The material of this Subsection is based on Cabella and Kamionkowski (2005), Challinor (2004), Challinor (2009), Challinor and Peiris (2009), Durrer (2008), and Lin and Wandelt (2006).

The probabilistic model of the CMB is introduced in Subsection~\ref{sub:probabilistic}. We define the set of vector bundles $\xi_s=(\mathcal{E}_s,\pi,S^2)$, $s\in\mathbb{Z}$, where the representation of the rotation group $G=SO(3)$ induced by the representation $W(g_{\alpha})=e^{\mathrm{i}s\alpha}$ of the subgroup $K=SO(2)$ is realised. In particular, the absolute temperature of the $CMB$, $T(\mathbf{n})$, is a single realisation of a mean square continuous strict sense isotropic (i.e., $SO(3)$-invariant) random field in $\xi_0$, while the complex polarisation, $(Q\pm\mathrm{i}U)(\mathbf{n})$, is a single realisation of a mean square continuous strict sense isotropic random field in $\xi_{\pm 2}$. Because any second order strict sense isotropic random field is automatically wide sense isotropic, Theorem~\ref{th:2} immediately gives the spectral decomposition of the above random fields. In the case of the absolute temperature, the functions \eqref{eq:spherical} become familiar \emph{spherical harmonics}, $Y_{\ell m}$, while in the case of the complex polarisation they become \emph{spin-weighted spherical harmonics}, ${}_{\pm 2}Y_{\ell m}$. This fact explains our notation, ${}_W\mathbf{Y}_{Vm}(t)$. The expansion coefficients are uncorrelated random variables with finite variance, which does not depend on the index $m$. In physical terms, the variance as a function of the parameter $\ell$ is the \emph{power spectrum}.

While studying physical literature, we have found that there exist various definitions of both ordinary and spin-weighted spherical harmonics. The choice of a definition is called the \emph{phase convention}. In terms of the representation theory, the phase convention is the choice of a basis in the space of the group representation. We made an attempt to describe different phase conventions in order to help the mathematicians to read physical literature. We also describe different notations for power spectra.

Following Zaldarriaga and Seljak (1997), we construct the random fields $E(\mathbf{n})$ and $B(\mathbf{n})$. The advantage of this fields over the complex polarisation fields $(Q\pm\mathrm{i}U)(\mathbf{n})$ is that the former fields are scalar (i.e., live in $\xi_0$), real-valued, and isotropic. Moreover, only $T(\mathbf{n})$ and $E(\mathbf{n})$ may be correlated, while two remaining pairs are always uncorrelated.

Our new result is Theorem~\ref{th:4}. It states that the standard assumptions of cosmological theories (the random fields $T(\mathbf{n})$, $E(\mathbf{n})$, and $B(\mathbf{n})$ are jointly isotropic) is equivalent to the assumption that $((Q-iU)(\mathbf{n}),T(\mathbf{n}),(Q+iU)(\mathbf{n}))$ is an isotropic random field in $\xi_{-2}\oplus\xi_0\oplus\xi_2$.

We conclude by two short remarks concerning Gaussian cosmological theories and an alternative description of the CMB in terms of the so called \emph{tensor spherical harmonics}.

Note that we do not consider questions connected with statistical analysis of the observation data of the recent and forthcoming experiments. For an introduction to this field of research, see Geller et al (2009) and the references herein.

I am grateful to Professor Domenico Marinucci for useful discussions on cosmology.

\section{Random fields in vector bundles}\label{sec:fields}

\subsection{Definitions}\label{sub:definitions}

Let $(\Omega,\mathfrak{F},\mathsf{P})$ be a probability space and let $\mathbf{X}(t)=\mathbf{X}(t,\omega)$ be a vector random field in a finite-dimensional $\mathbb{K}$-vector bundle $\xi=(\mathcal{E},\pi,T)$.

\begin{definition}
Let $X(t,\mathbf{x})$ be the scalar random field on the total space $\mathcal{E}$, defined as
\[
X(t,\mathbf{x})=(\mathbf{x},\mathbf{X}(t))_{\pi^{-1}(t)},
\qquad t\in T,\quad\mathbf{x}\in\pi^{-1}(t).
\]
We call $X(t,\mathbf{x})$ the scalar random field \emph{associated} to the vector random field $\mathbf{X}(t)$.
\end{definition}

The field $X(t,\mathbf{x})$ has the following property: its restriction onto $\pi^{-1}(t)$ is linear, i.e., for any $\mathbf{x}$, $\mathbf{y}\in\pi^{-1}(t)$, and for any $\alpha$, $\beta\in\mathbb{K}$,

\begin{equation}\label{eq:linear}
X(t,\alpha\mathbf{x}+\beta\mathbf{y})=\alpha X(t,\mathbf{x})+\beta X(t,\mathbf{y})\quad\mathsf{P}\text{-a.s.}
\end{equation}

\begin{definition}
A vector random field $\mathbf{X}(t)$ is \emph{mean square continuous} if the associated scalar random field $X(t,\mathbf{x})$ is mean square continuous, i.e., if the map
\[
\mathcal{E}\to L^2_{\mathbb{K}}(\Omega,\mathfrak{F},\mathsf{P}),\qquad (t,\mathbf{x})\mapsto X(t,\mathbf{x})
\]
is continuous.
\end{definition}

Let $H$ be a finite-dimensional $\mathbb{K}$-vector space with an inner product $(\boldsymbol{\cdot},\boldsymbol{\cdot})$. For any $\mathbf{x}\in H$, let $\mathbf{x}^*$ be the unique element of the conjugate space $H^*$ satisfying
\[
\mathbf{x}^*(\mathbf{y})=(\mathbf{y},\mathbf{x}),\qquad\mathbf{y}\in H.
\]

The mean value
\[
\mathbf{M}(t)=\mathsf{E}[\mathbf{X}(t)]
\]
of the mean square continuous random field $\mathbf{X}(t)$ is the continuous section of the vector bundle $\xi$, while its covariance operator
\[
R(s,t)=\mathsf{E}[\mathbf{X}(s)\otimes\mathbf{X}(t)^*]
\]
is the continuous section of the vector bundle $\xi\otimes\xi^*$.

Because the scalar random field $X(\mathbf{x})$ has property \eqref{eq:linear}, it can be left invariant with respect to the associated action, only if the restriction of the associated action onto any fiber $\pi^{-1}(t)$ is a linear invertible operator acting between the fibers. Moreover, the associated action must map the fiber $\pi^{-1}(t)$ onto the fiber $\pi^{-1}(gt)$.

\begin{definition}\label{def:associated}
Let $\xi=(\mathcal{E},\pi,T)$ be a vector bundle, and let $G\times T\to T$ be a continuous left action of a topological group $G$ on the base space $T$. A continuous left action $G\times\mathcal{E}\to\mathcal{E}$ of $G$ on the total space $\mathcal{E}$ is called \emph{associated} with the action $G\times T\to T$, if its restriction on any fiber $\pi^{-1}(t)$ is an invertible linear operator acting from $\pi^{-1}(t)$ to $\pi^{-1}(gt)$.
\end{definition}

We are ready to formulate the main definitions of Subsection~\ref{sub:definitions}.

\begin{definition}\label{def:wide}
Let $\xi=(\mathcal{E},\pi,T)$ be a vector bundle, let $G\times T\to T$ be a continuous left action of a topological group $G$ on the base space $T$, and let $G\times\mathcal{E}\to\mathcal{E}$ be an associated action of $G$ on the total space $\mathcal{E}$. A vector random field $\mathbf{X}(t)$ on $\xi$ is called \emph{wide sense left $G$-invariant} if the associated scalar random field $X(t,\mathbf{x})$ is wide sense left invariant with respect to the associated action $G\times\mathcal{E}\to\mathcal{E}$, i.e., for all $g\in G$, for all $s$, $t\in T$, and for all $\mathbf{x}\in\pi^{-1}(s)$, $\mathbf{y}\in\pi^{-1}(t)$ we have
\[
\begin{aligned}
\mathsf{E}[X(gs,g\mathbf{x})]&=\mathsf{E}[X(s,\mathbf{x})],\\
\mathsf{E}[X(gs,g\mathbf{x})\overline{X(gt,g\mathbf{y})}]&=
\mathsf{E}[X(s,\mathbf{x})\overline{X(t,\mathbf{y})}].
\end{aligned}
\]
\end{definition}

\begin{definition}
Under conditions of Definition~\ref{def:wide}, a vector random field $\mathbf{X}(t)$ on $\xi$ is called \emph{strict sense left $G$-invariant} if the associated scalar random field $X(t,\mathbf{x})$ is strict sense left invariant with respect to the associated action $G\times\mathcal{E}\to\mathcal{E}$, i.e., all finite-dimensional distributions of the random field $X(t,\mathbf{x})$ are invariant under the associated action.
\end{definition}

It is easy to see that any mean square continuous strict sense invariant random field is wide sense invariant. On the other hand, any Gaussian wide sense invariant random field is strict sense invariant.

\subsection{An example of associated action}\label{sub:associated}

Let $G$ be a topological group, and let $K$ be its closed subgroup. Let $T$ be the homogeneous space $G/K$ of \emph{left} cosets $g_0K$, $g_0\in G$. An element $g\in G$ acts on $T$ by left multiplication:

\begin{equation}\label{eq:action}
g_0K\mapsto gg_0K.
\end{equation}

Let $W$ be a representation of $K$ on a finite-dimensional complex Hilbert space $H$. Consider the following action of $K$ on the Cartesian product $G\times H$:
\[
k(g,\mathbf{x})=(gk,W(k^{-1})\mathbf{x}).
\]
Denote the quotient space of orbits of the above action by $\mathcal{E}_W$. The projection
\[
\pi\colon \mathcal{E}_W\to T,\qquad\pi(g,\mathbf{x})=gK
\]
determines the \emph{homogeneous, or equivariant} vector bundle $\xi=(\mathcal{E}_W,\pi,T)$.

Let $t=g_0K\in T$. It is trivial to check that the action
\[
g(g_0K,\mathbf{x})=(gg_0K,\mathbf{x})
\]
is associated to the action \eqref{eq:action}.

Moreover, let $\mathbf{X}(t)$ be a random field in $\xi$. All random vectors $\mathbf{X}(t)$ lie in the same space $H$. By definition, the associated scalar random field $X(t,\mathbf{x})=(\mathbf{x},\mathbf{X}(t))$ is mean square continuous if and only if
\[
\lim_{(s,\mathbf{y})\to(t,\mathbf{x})}\mathsf{E}|X(s,\mathbf{y})
-X(t,\mathbf{x})|^2=0.
\]
Let $\{\mathbf{e}_1,\mathbf{e}_2,\dots,\mathbf{e}_{\dim H}\}$ be a basis in $H$. Put $\mathbf{y}=\mathbf{x}=\mathbf{e}_j$. Then we have
\[
\lim_{s\to t}\mathsf{E}|\overline{X_j(s)}-\overline{X_j(t)}|^2=0,
\]
which is equivalent to
\[
\lim_{s\to t}\mathsf{E}|X_j(s)-X_j(t)|^2=0.
\]
It follows that
\[
\begin{aligned}
\lim_{s\to t}\mathsf{E}\|\mathbf{X}(s)-\mathbf{X}(t)\|^2&=
\lim_{s\to t}\mathsf{E}\sum^{\dim H}_{j=1}|X_j(s)-X_j(t)|^2\\
&=\sum^{\dim H}_{j=1}\lim_{s\to t}\mathsf{E}|X_j(s)-X_j(t)|^2=0.
\end{aligned}
\]

Conversely, let
\begin{equation}\label{eq:common}
\lim_{s\to t}\mathsf{E}\|\mathbf{X}(s)-\mathbf{X}(t)\|^2=0.
\end{equation}
Then, for any $j=1$, $2$, \dots, $\dim H$,
\[
\begin{aligned}
0&\leq\limsup_{s\to t}\mathsf{E}|X_j(s)-X_j(t)|^2\\
&\leq\sum^{\dim H}_{j=1}\limsup_{s\to t}\mathsf{E}|X_j(s)-X_j(t)|^2=0,
\end{aligned}
\]
thus, $\lim_{s\to t}\mathsf{E}|X_j(s)-X_j(t)|^2=0$. It follows that
\[
\begin{aligned}
\lim_{(s,\mathbf{y})\to(t,\mathbf{x})}\mathsf{E}|X(s,\mathbf{y})
-X(t,\mathbf{x})|^2&=\lim_{(s,\mathbf{y})\to(t,\mathbf{x})}\mathsf{E}
\left|\sum^{\dim H}_{j=1}(y_j\overline{X_j(s)}-x_j\overline{X_j(t)})\right|^2\\
&\leq 2\sum^{\dim H}_{j=1}\lim_{(s,\mathbf{y})\to(t,\mathbf{x})}\mathsf{E}
|y_jX_j(s)-x_jX_j(t)|^2=0.
\end{aligned}
\]

We proved that in the particular case of a vector random field in a homogeneous vector bundle our definition of mean square continuity is equivalent to the usual definition \eqref{eq:common}. In the same way one can easily prove that our definition of a wide sense $G$-invariant field is equivalent to the following equalities: for all $s$, $t\in T$, and for all $g\in G$ we have
\begin{equation}\label{eq:invariant}
\begin{aligned}
\mathsf{E}[\mathbf{X}(gs)]&=\mathsf{E}[\mathbf{X}(s)],\\
\mathsf{E}[\mathbf{X}(gs)\otimes\mathbf{X}^*(gt)]
&=\mathsf{E}[\mathbf{X}(s)\otimes\mathbf{X}^*(t)].
\end{aligned}
\end{equation}
The first equation is equivalent to the following equality
\[
\mathsf{E}[\mathbf{X}(s)]=\mathsf{E}[\mathbf{X}(t)],\qquad s,t\in T,
\]
because $G$ acts transitively on $T$. Thus, the mean value of a wide sense $G$-invariant random field on $T$ is constant.

\subsection{The spectral decomposition of a vector random field over a compact homogeneous space}\label{sub:spectral}

Let $G$ be a compact topological group, and let $K$ be its closed subgroup. Let $T$ be the homogeneous space $G/K$. Let $W$ be a representation of $K$ on a finite-dimensional complex Hilbert space $H$, and let $\xi=(\mathcal{E}_W,\pi,T)$ be the corresponding homogeneous vector bundle. Let $\hat{G}$ (resp. $\hat{K}$) be the set of all equivalence classes of irreducible unitary representations of $G$ (resp. $K$). For simplicity, assume that $K$ is \emph{massive} in $G$ (Vilenkin, 1968). This means that for all $V\in\hat{G}$ and for all $W\in\hat{K}$ the multiplicity of $W$ in the restriction of $V$ onto $K$ is either $0$ or $1$.

First, consider the case when $W$ is an \emph{irreducible} unitary representation of $K$. Let $dg$ be the Haar measure on $G$ with $\int_G\,dg=1$. Let $L^2(G,H)$ be the set of all measurable functions $\mathbf{f}\colon G\to H$ such that
\[
\int_G\|\mathbf{f}(g)\|^2\,\mathrm{d}g<\infty
\]
and
\begin{equation}\label{eq:localrotation}
\mathbf{f}(gk)=W(k^{-1})\mathbf{f}(g),\qquad g\in G,\quad k\in K.
\end{equation}
To each $\mathbf{f}\in L^2(G,H)$, we associate the map $\mathbf{s}\colon T\to \mathcal{E}_W$: $\mathbf{s}(gK)=(g,\mathbf{f}(g))$. The above association is an isomorphism between $L^2(G,H)$ and the space $L^2(\mathcal{E}_W)$ of the square integrable sections of the homogeneous vector bundle $\xi$. This space can be considered as a space of ``twisted" functions on the base space $T$. If $W$ is the trivial representation of $K$ in $H=\mathbb{C}$, then we return back to the standard space $L^2(G)$. The representation
\[
[\mathcal{U}(g)\mathbf{s}](t)=\mathbf{s}(g^{-1}t)
\]
is the representation of $G$ induced from the representation~$W$ of the subgroup~$K$.

We need the following facts about induced representations.

\begin{enumerate}

\item Frobenius reciprocity: the multiplicity of $V\in\hat{G}$ in $\mathcal{U}$ is equal to the multiplicity of $W$ in $V$.

\item The representation induced from the direct sum $W_1\oplus W_2\oplus\cdots\oplus W_N$ is the direct sum of representations induced from $W_1$, $W_2$, \dots, $W_N$.

\end{enumerate}

Let $\hat{G}_K(W)$ be the set of all $V\in\hat{G}$ whose restrictions onto $K$ contain $W$ (necessarily once, because $K$ is massive in $G$). For any $V\in\hat{G}_K(W)$, let $i_V$ be the embedding of $H$ into the space $H_V$ of the representation $V$. Let $p_V$ be the orthogonal projection from $H_V$ onto $H$. By the result of Camporesi (2005), any $\mathbf{f}\in L^2(G,H)$ can be represented by the series
\[
\mathbf{f}(g)=\frac{1}{\dim W}\sum_{V\in\hat{G}_K(W)}\dim V\int_Gp_VV(g^{-1}h)i_V\mathbf{f}(h)\,\mathrm{d}h.
\]
The above series converges in strong topology of the Hilbert space $L^2(G,H)$, i.e.,
\[
\|\mathbf{f}\|^2_{L^2(G,H)}=\frac{1}{\dim W}\sum_{V\in\hat{G}_K(W)}\dim V\int_G\left(\int_Gp_VV(g^{-1}h)i_V\mathbf{f}(h)\,\mathrm{d}h,\mathbf{f}(g)\right)\,\mathrm{d}g.
\]

Fix a basis $\{\mathbf{e}_1,\mathbf{e}_2,\dots,\mathbf{e}_{\dim H}\}$ of the space $H$. Let $\{\mathbf{e}_1^{(V)},\mathbf{e}_2^{(V)},\dots,\mathbf{e}_{\dim H_V}^{(V)}\}$ be a basis in $H_V$ with
\begin{equation}\label{eq:embedding}
i_V\mathbf{e}_j=\mathbf{e}_{p+j}^{(V)}\qquad 1\leq j\leq\dim W.
\end{equation}
for some $p\geq 0$. Let $f_j(g)=(\mathbf{f}(g),\mathbf{e}_j)$ be the coordinates of $\mathbf{f}(g)$. Equation~\eqref{eq:embedding} means that $W$ acts in the linear span of the $\dim W$ basis vectors of $H_V$ that are enumerated without lacunas. Then we have
\[
i_V\mathbf{f}(h)=(0,\dots,0,f_1(h),\dots,f_{\dim W}(h),0,\dots,0).
\]
Let $V_{m,n}(g)=(V(g)\mathbf{e}_m^{(V)},\mathbf{e}_n^{(V)})$ be the matrix elements of the representation $V$. Then
\[
(V(g^{-1}h)i_V\mathbf{f}(h))_{p+j}=\sum^{\dim V}_{m=1}\overline{V_{m,p+j}(g)}\sum^{\dim W}_{n=1}V_{m,p+n}(h)f_n(h)
\]
and
\[
f_j(g)=\frac{1}{\dim W}\sum_{V\in\hat{G}_K(W)}\dim V\sum^{\dim V}_{m=1}\sum^{\dim W}_{n=1}\int_Gf_n(h)V_{m,p+n}(h)\,\mathrm{d}h\overline{V_{m,p+j}(g)}.
\]

From now, let $G$ be a connected compact Lie group, and let $p\colon G\to T$ denote a natural projection: $p(g)=gK$. Let $D_G$ be an open dense subset in $G$, and let $(D_G,\mathbf{J}(g))$ with
\[
\mathbf{J}(g)=(\theta_1(g),\dots,\theta_{\dim G}(g))\colon D_G\to\mathbb{R}^{\dim G}
\]
be a chart of the atlas of the manifold $G$ with the following property: if $k\in K$ and both $g$ and $kg$ lie in $D_G$, then $\theta_j(kg)=\theta_j(g)$ for $1\leq j\leq\dim T$. Then,  $(D_T,\mathbf{I}(t))$ with
\begin{equation}\label{eq:map}
\begin{aligned}
D_T&=pD_G,\\
\mathbf{I}(t)&=(\theta_1(t),\dots,\theta_{\dim T}(t))\colon D_T\to\mathbb{R}^{\dim T}
\end{aligned}
\end{equation}
is a chart of the atlas of the manifold $T$, and the domain $D_T$ of this chart is dense in $T$. Let $t\in D_T$ has local coordinates $(\theta_1,\dots,\theta_{\dim T})$ in the chart \eqref{eq:map}. Then, the representation of the section $\mathbf{s}\in L^2(\mathcal{E}_W)$ associated to $\mathbf{f}\in L^2(G,H)$ has the form
\begin{equation}\label{eq:section}
s_j(t)=\frac{1}{\dim W}\sum_{V\in\hat{G}_K(W)}\dim V\sum^{\dim V}_{m=1}\sum^{\dim W}_{n=1}\int_Ts_n(t)V_{m,p+n}(t)\,\mathrm{d}t\overline{V_{m,p+j}(t)}
\end{equation}
where $dt$ is the $G$-invariant measure on $T$ with $\int_T\,\mathrm{d}t=1$, and
\[
V_{m,p+n}(t)=V_{m,p+n}(\theta_1,\dots,\theta_{\dim T},\theta_{\dim T+1}^{(0)},\dots,\theta_{\dim G}^{(0)}).
\]

Introduce the following notation:
\begin{equation}\label{eq:spherical}
{}_W\mathbf{Y}_{Vm}(t)=\sqrt{\frac{\dim V}{\dim W}}(\overline{V_{m,p+1}(t)},\overline{V_{m,p+2}(t)},\dots,\overline{V_{m,p+\dim W}(t)}).
\end{equation}
Note that the correct notation must be ${}_W\mathbf{Y}_{\mathbf{I}Vm}(t)$, because functions \eqref{eq:spherical} depend on the choice of a chart. In what follows, we use only chart \eqref{eq:map} and suppress symbol $\mathbf{I}$ for notational simplicity.

Equation~\eqref{eq:section} means that the functions $\{\,{}_W\mathbf{Y}_{Vm}(t)\colon V\in\hat{G}_K(W),1\leq m\leq\dim V\,\}$ form a basis in $L^2(\mathcal{E}_W)$, i.e.,
\begin{equation}\label{eq:deterministic}
s_j(t)=\sum_{V\in\hat{G}_K(W)}\sum^{\dim V}_{m=1}\sum^{\dim W}_{n=1}\int_Ts_n(t)\overline{({}_WY_{Vm})_n(t)}\,\mathrm{d}t({}_WY_{Vm})_j(t)
\end{equation}

Let $\mathbf{X}(t)$ be a mean square continuous random field in $\xi$. Consider the following random variables:
\begin{equation}\label{eq:Z}
Z_{mn}^{(V)}=\int_TX_n(t)\overline{({}_WY_{Vm})_n(t)}\,\mathrm{d}t,
\end{equation}
where $V\in\hat{G}_K(W)$, $1\leq m\leq\dim V$, and $1\leq n\leq\dim W$. This integral has to be understood as a Bochner integral of a function taking values in the space $L^2_{\mathbb{K}}(\Omega,\mathfrak{F},\mathsf{P})$.

\begin{theorem}\label{th:1}
Let $G$ be a connected compact Lie group, let $K$ be its massive subgroup, let $W$ be an irreducible unitary representation of the group~$K$, and let $\xi$ be the corresponding homogeneous vector bundle. In the chart \eqref{eq:map}, a mean square continuous random field $\mathbf{X}(t)$ in $\xi$ has the form
\begin{equation}\label{eq:spectral}
X_j(t)=\sum_{V\in\hat{G}_K(W)}\sum^{\dim V}_{m=1}\sum^{\dim W}_{n=1}Z_{mn}^{(V)}({}_WY_{Vm})_j(t),
\end{equation}
where random variables $Z_{mn}^{(V)}$ have the form \eqref{eq:Z}.
\end{theorem}

\begin{proof}
Let $\mathbf{M}(t)$ be the mean value of the random field $\mathbf{X}(t)$, and let $R(t_1,t_2)$ be its covariance operator. Denote the right hand side of \eqref{eq:spectral} by $Z_j(t)$. We need to prove that
\[
\mathsf{E}[\mathbf{Z}(t)]=\mathbf{M}(t)
\]
and
\[
\mathsf{E}[\mathbf{Z}(t_1)\otimes\mathbf{Z}^*(t_2)]=R(t_1,t_2).
\]

Using \eqref{eq:Z}, we obtain
\[
\mathsf{E}[Z_{mn}^{(V)}]=\int_T\mathsf{E}[X_n(t)]\overline{({}_WY_{Vm})_n(t)}\,\mathrm{d}t.
\]
It follows that
\begin{equation}\label{eq:mean}
\begin{aligned}
\mathsf{E}[Z_j(t)]&=\mathsf{E}\left[\sum_{V\in\hat{G}_K(W)}\sum^{\dim V}_{m=1}\sum^{\dim W}_{n=1}Z_{mn}^{(V)}({}_WY_{Vm})_j(t)\right]\\
&=\sum_{V\in\hat{G}_K(W)}\sum^{\dim V}_{m=1}\sum^{\dim W}_{n=1}\mathsf{E}[Z_{mn}^{(V)}]({}_WY_{Vm})_j(t)\\
&=\sum_{V\in\hat{G}_K(W)}\sum^{\dim V}_{m=1}\sum^{\dim W}_{n=1}\int_T\mathsf{E}[X_n(t)]\overline{({}_WY_{Vm})_n(t)}\,\mathrm{d}t({}_WY_{Vm})_j(t)\\
&=\mathsf{E}[X_j(t)]
\end{aligned}
\end{equation}
by \eqref{eq:deterministic}.

Similarly,
\[
\mathsf{E}[Z_{mn}^{(V)}\overline{Z_{m'n'}^{(V')}}]=\iint_{T\times T}
R(t_1,t_2)\overline{({}_WY_{Vm})_n(t_1)}({}_WY_{V'm'})_{n'}(t_2)\,\mathrm{d}t_1\,\mathrm{d}t_2.
\]
It follows that
\begin{equation}\label{eq:variance}
\begin{aligned}
\mathsf{E}[Z_j(t_1)\overline{Z_{j'}(t_2)}]&=\sum_{V,V'\in\hat{G}_K(W)}\sum^{\dim V}_{m=1}\sum^{\dim V'}_{m'=1}\sum^{\dim W}_{n,n'=1}\mathsf{E}[Z_{mn}^{(V)}\overline{Z_{m'n'}^{(V')}}]({}_WY_{Vm})_j(t_1)
\overline{({}_WY_{V'm'})_{j'}(t_2)}\\
&=\sum_{V,V'\in\hat{G}_K(W)}\sum^{\dim V}_{m=1}\sum^{\dim V'}_{m'=1}\sum^{\dim W}_{n,n'=1}\iint_{T\times T}R(t_1,t_2)\overline{({}_WY_{Vm})_n(t_1)}\\
&\quad\times({}_WY_{V'm'})_{n'}(t_2)\,\mathrm{d}t_1\,\mathrm{d}t_2({}_WY_{Vm})_j(t_1)\overline{({}_WY_{V'm'})_{j'}(t_2)}\\
&=R_{jj'}(t_1,t_2).
\end{aligned}
\end{equation}
\end{proof}

Denote by $V_0$ the trivial irreducible representation of the group $G$.

\begin{theorem}\label{th:2}
Under conditions of Theorem~\ref{th:1}, the following statements are equivalent.
\begin{enumerate}

\item $\mathbf{X}(t)$ is a mean square continuous wide sense invariant random field in $\xi$.

\item $\mathbf{X}(t)$ has the form \eqref{eq:spectral}, where $Z_{mn}^{(V)}$, $V\in\hat{G}_K(W)$, $1\leq m\leq\dim V$, $1\leq n\leq\dim W$ are random variables satisfying the following conditions.

    \begin{itemize}

    \item If $V\neq V_0$, then $\mathsf{E}[Z_{mn}^{(V)}]=0$ .

    \item $\mathsf{E}[Z_{mn}^{(V)}\overline{Z_{m'n'}^{(V')}}]=\delta_{VV'}
        \delta_{mm'}R^{(V)}_{nn'}$, with
        \begin{equation}\label{eq:convergent}
        \sum_{V\in\hat{G}_K(W)}\dim V\tr[R^{(V)}]<\infty.
        \end{equation}

    \end{itemize}

\end{enumerate}
\end{theorem}

\begin{proof}
Let $\mathbf{X}(t)$ be a mean square continuous wide sense invariant random field in $\xi$. By Theorem~\ref{th:1}, $\mathbf{X}(t)$ has the form \eqref{eq:spectral}. By \eqref{eq:mean}, we have
\[
\mathsf{E}[X_j(t)]=\sum_{V\in\hat{G}_K(W)}\sum^{\dim V}_{m=1}\sum^{\dim W}_{n=1}\mathsf{E}[Z_{mn}^{(V)}]({}_WY_{Vm})_j(t).
\]
Let $g\in G$. Substitute $gt$ in place of $t$ to the last display. We obtain
\[
\begin{aligned}
\mathsf{E}[X_j(gt)]&=\sum_{V\in\hat{G}_K(W)}\sum^{\dim V}_{m=1}\sum^{\dim W}_{n=1}\mathsf{E}[Z_{mn}^{(V)}]({}_WY_{Vm})_j(gt)\\
&=\sum_{V\in\hat{G}_K(W)}\sum^{\dim V}_{m=1}\sum^{\dim W}_{n=1}\mathsf{E}[Z_{mn}^{(V)}]\sum^{\dim V}_{\ell=1}
\overline{V_{m\ell}(g)}({}_WY_{V\ell})_j(t)\\
&=\sum_{V\in\hat{G}_K(W)}\sum^{\dim V}_{m=1}\sum^{\dim W}_{n=1}\sum^{\dim V}_{\ell=1}\overline{V_{\ell m}(g)}
\mathsf{E}[Z_{\ell n}^{(V)}]({}_WY_{Vm})_j(t).
\end{aligned}
\]
The left hand sides of the two last displays are equal. Therefore, the coefficients of the expansions must be equal.
\[
\sum^{\dim V}_{\ell=1}\overline{V_{\ell m}(g)}
\mathsf{E}[Z_{\ell n}^{(V)}]=\mathsf{E}[Z_{mn}^{(V)}]
\]
Denote $\mathbf{M}_n^{(V)}=(\mathsf{E}[Z_{mn}^{(V)}],\dots,\mathsf{E}[Z_{\dim Vn}^{(V)}])$. Then
\[
V^+(g)\mathbf{M}_n^{(V)}=\mathbf{M}_n^{(V)},\qquad g\in G,
\]
where $V^+(g)=V(g^{-1})^{\top}$ is the representation, dual to the representation $V$. It follows that either $\mathbf{M}_n^{(V)}=\mathbf{0}$ or the one-dimensional subspace generated by $\mathbf{M}_n^{(V)}$ is an invariant subspace of the irreducible representation $V^+$. In the latter case, $V^+$ must be one-dimensional. If $V^+$ is trivial, then $V$ is also trivial, and $\mathbf{M}_n^{(V)}$ is any complex number. If $V^+$ is not trivial, so is $V$. Then, there exist $g\in G$ with $V(g)\neq 1$. It follows that $\mathbf{M}_n^{(V)}=V(g)\mathbf{M}_n^{(V)}=\mathbf{0}$.

By \eqref{eq:variance}, we have
\[
R_{jj'}(t_1,t_2)=\sum_{V,V'\in\hat{G}_K(W)}\sum^{\dim V}_{m=1}\sum^{\dim V'}_{m'=1}\sum^{\dim W}_{n,n'=1}\mathsf{E}[Z_{mn}^{(V)}\overline{Z_{m'n'}^{(V')}}]({}_WY_{Vm})_j(t_1)
\overline{({}_WY_{V'm'})_{j'}(t_2)}.
\]
It follows that
\[
\begin{aligned}
R_{jj'}(gt_1,gt_2)&=\sum_{V,V'\in\hat{G}_K(W)}\sum^{\dim V}_{m=1}\sum^{\dim V'}_{m'=1}\sum^{\dim W}_{n,n'=1}\mathsf{E}[Z_{mn}^{(V)}\overline{Z_{m'n'}^{(V')}}]({}_WY_{Vm})_j(gt_1)
\overline{({}_WY_{V'm'})_{j'}(gt_2)}\\
&=\sum_{V,V'\in\hat{G}_K(W)}\sum^{\dim V}_{m=1}\sum^{\dim V'}_{m'=1}\sum^{\dim W}_{n,n'=1}\mathsf{E}[Z_{mn}^{(V)}\overline{Z_{m'n'}^{(V')}}]\\
&\quad\times\sum^{\dim V}_{\ell=1}\overline{V_{m\ell}(g)}({}_WY_{V\ell})_j(t_1)
\sum^{\dim V'}_{\ell'=1}V'_{m'\ell'}(g)\overline{({}_WY_{V'\ell'})_{j'}(t_2)}\\
&=\sum_{V,V'\in\hat{G}_K(W)}\sum^{\dim V}_{m=1}\sum^{\dim V'}_{m'=1}\sum^{\dim W}_{n,n'=1}\sum^{\dim V}_{\ell=1}\sum^{\dim V'}_{\ell'=1}\overline{V_{m\ell}(g)}V'_{m'\ell'}(g)\\
&\quad\times\mathsf{E}[Z_{\ell n}^{(V)}\overline{Z_{\ell'n'}^{(V')}}]({}_WY_{Vm})_j(t_1)
\overline{({}_WY_{V'm'})_{j'}(t_2)}.
\end{aligned}
\]
By equating the coefficients of the two expansions, we obtain
\[
\sum^{\dim V}_{\ell=1}\sum^{\dim V'}_{\ell'=1}\overline{V_{m\ell}(g)}V'_{m'\ell'}(g)\mathsf{E}[Z_{\ell n}^{(V)}\overline{Z_{\ell'n'}^{(V')}}]=\mathsf{E}[Z_{mn}^{(V)}\overline{Z_{m'n'}^{(V')}}].
\]
Let $P^{(V,V')}_{nn'}$ be the matrix with elements
\[
(P^{(V,V')}_{nn'})_{mm'}=\mathsf{E}[Z_{mn}^{(V)}\overline{Z_{m'n'}^{(V')}}].
\]
Then,
\[
(V^+\otimes V')(g)P^{(V,V')}_{nn'}=P^{(V,V')}_{nn'},\qquad g\in G.
\]
It follows that either $P^{(V,V')}_{nn'}$ is zero matrix or the one-dimensional subspace generated by $P^{(V,V')}_{nn'}$ is an invariant subspace of the representation $V^+\otimes V'$. In the latter case, the representation $V^+\otimes V'$ contains an one-dimensional irreducible component, say $\mathcal{V}$. If $\mathcal{V}$ is trivial, then $V=V'$ and $\mathcal{V}$ acts in the one-dimensional subspace generated by the unit matrix $(P^{(V,V')}_{nn'})_{mm'}$. If $\mathcal{V}$ is not trivial, there exist $g\in G$ with $\mathcal{V}(g)\neq 1$. It follows that $P^{(V,V')}_{nn'}=\mathcal{V}(g)P^{(V,V')}_{nn'}$, so $P^{(V,V')}_{nn'}$ is zero matrix. So,
\[
\mathsf{E}[Z_{mn}^{(V)}\overline{Z_{m'n'}^{(V')}}]=\delta_{VV'}
\delta_{mm'}R^{(V)}_{nn'}.
\]

Let $t_0\in T$ be the left coset of the unit element of $G$. We may assume $t_0\in D_T$ (otherwise use a chart $(gD_T,\mathbf{I}(g^{-1}t))$ for a suitable $g\in G$). Then
\[
({}_WY_{Vm})_j(t_0)=\sqrt{\frac{\dim V}{\dim W}}\delta_{m,p+j}
\]
and
\[
\begin{aligned}
X_j(t_0)&=\sum_{V\in\hat{G}_K(W)}\sum^{\dim V}_{m=1}\sum^{\dim W}_{n=1}Z_{mn}^{(V)}({}_WY_{Vm})_j(t)\\
&=\frac{1}{\sqrt{\dim W}}\sum_{V\in\hat{G}_K(W)}\sqrt{\dim V}\sum^{\dim W}_{n=1}Z_{jn}^{(V)}.
\end{aligned}
\]
It follows that
\[
\begin{aligned}
\mathsf{E}|X_j(t_0)|^2&=\frac{1}{\dim W}\sum_{V\in\hat{G}_K(W)}\dim V\dim W\mathsf{E}|Z_{j1}^{(V)}|^2\\
&=\sum_{V\in\hat{G}_K(W)}\dim V R^{(V)}_{jj},
\end{aligned}
\]
and
\[
\sum_{V\in\hat{G}_K(W)}\dim V\tr[R^{(V)}]=\sum^{\dim W}_{j=1}\mathsf{E}|X_j(t_0)|^2<\infty.
\]

Conversely, let $Z_{mn}^{(V)}$, $V\in\hat{G}_K(W)$, $1\leq m\leq\dim V$, $1\leq n\leq\dim W$ be random variables satisfying conditions of Theorem~\ref{th:2}. Consider random field \eqref{eq:spectral}. Then, its mean value is
\[
E[X_j(t)]=
\begin{cases}
\mathsf{E}[Z^{(V_0)}_{11}],&V_0\in\hat{G}_K(W),\\
0,&\text{otherwise},
\end{cases}
\]
which is constant. Note that $V_0\in\hat{G}_K(W)$ if and only if $W$ is trivial (by Frobenius reciprocity).

The correlation operator of the random field \eqref{eq:spectral} is
\[
\begin{aligned}
R_{jj'}(t_1,t_2)&=\sum_{V,V'\in\hat{G}_K(W)}\sum^{\dim V}_{m=1}\sum^{\dim V'}_{m'=1}\sum^{\dim W}_{n,n'=1}\mathsf{E}[Z_{mn}^{(V)}\overline{Z_{m'n'}^{(V')}}]({}_WY_{Vm})_j(t_1)
\overline{({}_WY_{V'm'})_{j'}(t_2)}\\
&=\sum_{V\in\hat{G}_K(W)}\sum^{\dim W}_{n,n'=1}R^{(V)}_{nn'}\sum^{\dim V}_{m=1}({}_WY_{Vm})_j(t_1)
\overline{({}_WY_{Vm})_{j'}(t_2)}\\
&=\frac{1}{\dim W}\sum_{V\in\hat{G}_K(W)}\dim V\sum^{\dim W}_{n,n'=1}R^{(V)}_{nn'}V_{p+j,p+j'}(g_1^{-1}g_2),
\end{aligned}
\]
where $g_1$ (resp. $g_2$) is an arbitrary element from the left coset corresponding to $t_1$ (resp. $t_2$). The terms of this functional series are bounded by the terms of the convergent series \eqref{eq:convergent}, because $|V_{p+j,p+j'}(g_1^{-1}g_2)|\leq 1$. Therefore, the series converges uniformly, and its sum is continuous function. This means that $\mathbf{X}(t)$ is mean square continuous.

For any $g\in G$, we have
\[
\begin{aligned}
R_{jj'}(gt_1,gt_2)&=\frac{1}{\dim W}\sum_{V\in\hat{G}_K(W)}\dim V\sum^{\dim W}_{n,n'=1}R^{(V)}_{nn'}V_{p+j,p+j'}((gg_1)^{-1}gg_2)\\
&=R_{jj'}(t_1,t_2),
\end{aligned}
\]
so $\mathbf{X}(t)$ is invariant.
\end{proof}

Assume that $W$ is not necessarily irreducible representation of $K$ in a finite-dimensional complex Hilbert space $H$. Because $K$ is compact, the representation~$W$ is equivalent to a direct sum $W_1\oplus W_2\oplus\cdots\oplus W_N$ of irreducible unitary representations of $K$. The representation induced by $W$ is a direct sum of representations induced by $W_k$, $1\leq k\leq N$. It is realised in a homogeneous vector bundle $\xi=\xi_1\oplus\xi_2\oplus\cdots\oplus\xi_N$, where $\xi_k$ is the homogeneous vector bundle that carries the irreducible component $W_k$.

Let $\mathbf{X}(t)$ be an invariant random field in $\xi$. Denote the components of $\mathbf{X}(t)$ by $X^{(k)}_j(t)$, $1\leq k\leq N$, $1\leq j\leq\dim W_k$. Denote by $P_k$ the orthogonal projection from $H$ onto the space $H_k$ where the irreducible component $W_k$ acts.

\begin{theorem}\label{th:3}
Under conditions of Theorem~\ref{th:1}, the following statements are equivalent.

\begin{enumerate}

\item $\mathbf{X}(t)$ is a mean square continuous wide sense invariant random field in $\xi$.

\item $\mathbf{X}(t)$ has the form
\begin{equation}\label{eq:spectralreducible}
X^{(k)}_j(t)=\sum_{V\in\hat{G}_K(W_k)}\sum^{\dim V}_{m=1}\sum^{\dim W_k}_{n=1}Z^{(Vk)}_{mn}({}_{W_k}Y_{Vm})_j(t),
\end{equation}
where $Z_{mn}^{(Vk)}$, $1\leq k\leq N$, $V\in\hat{G}_K(W_k)$, $1\leq m\leq\dim V$, $1\leq n\leq\dim W_k$ are random variables satisfying the following conditions.

    \begin{itemize}

    \item If $V\neq V_0$, then $\mathsf{E}[Z_{mn}^{(Vk)}]=0$.

    \item $\mathsf{E}[Z_{mn}^{(Vk)}\overline{Z_{m'n'}^{(V'k')}}]=\delta_{VV'}
        \delta_{mm'}R^{(V)}_{kn,k'n'}$, with
        \[
        \sum^N_{k=1}\sum_{V\in\hat{G}_K(W_k)}\dim V\tr[P_kR^{(V)}P_k]<\infty.
        \]

    \end{itemize}

\end{enumerate}
\end{theorem}

\begin{proof}
Use mathematical induction. The induction base, when $N=1$, is Theorem~\ref{th:2}. Assume the induction hypotheses:  Theorem~\ref{th:3} is proved up to $N-1$.

Let $\mathbf{X}(t)$ be a mean square continuous wide sense invariant random field in $\xi$. Then the field
\[
\mathbf{Y}_1(t)=(X^{(1)}_1(t),\dots,X^{(1)}_{\dim W_1}(t),\dots,X^{(N-1)}_1(t),\dots,X^{(N-1)}_{\dim W_{N-1}}(t)),\\
\]
is a mean square continuous wide sense invariant random field in $\xi_1\oplus\cdots\oplus\xi_{N-1}$. By the induction hypotheses, \[
X^{(k)}_j(t)=\sum_{V\in\hat{G}_K(W_k)}\sum^{\dim V}_{m=1}\sum^{\dim W_k}_{n=1}Z^{(Vk)}_{mn}({}_{W_k}Y_{Vm})_j(t),\qquad 1\leq k\leq N-1,
\]
where $\mathsf{E}[Z_{mn}^{(Vk)}]=0$ unless $V\neq V_0$ and $\mathsf{E}[Z_{mn}^{(Vk)}\overline{Z_{m'n'}^{(V'k')}}]=\delta_{VV'}
\delta_{mm'}R^{(V,N-1)}_{kn,k'n'}$, with
\[
\sum^{N-1}_{k=1}\sum_{V\in\hat{G}_K(W_k)}\dim V\tr[P_kR^{(V,N-1)}P_k]<\infty.
\]
The field
\[
\mathbf{Y}_2(t)=(X^{(N)}_1(t),\dots,X^{(N)}_{\dim W_N}(t))
\]
is a mean square continuous wide sense invariant random field in $\xi_N$. By Theorem~\ref{th:2},
\[
X^{(N)}_j(t)=\sum_{V\in\hat{G}_K(W_N)}\sum^{\dim V}_{m=1}\sum^{\dim W_N}_{n=1}Z^{(VN)}_{mn}({}_{W_N}Y_{Vm})_j(t),
\]
where $\mathsf{E}[Z_{mn}^{(VN)}]=0$ unless $V\neq V_0$ and $\mathsf{E}[Z_{mn}^{(VN)}\overline{Z_{m'n'}^{(V'N)}}]=\delta_{VV'}
\delta_{mm'}R^{(VN)}_{nn'}$, with
\[
\sum_{V\in\hat{G}_K(W_N)}\dim V\tr[R^{(VN)}]<\infty.
\]
The matrix $R^{(V)}_{kn,k'n'}$ with elements
\[
R^{(V)}_{kn,k'n'}=\mathsf{E}[Z_{1n}^{(Vk)}\overline{Z_{1n'}^{(Vk')}}]
\]
obviously satisfies conditions of the second item of Theorem~\ref{th:3}.

Conversely, let $Z_{mn}^{(Vk)}$, $1\leq k\leq N$, $V\in\hat{G}_K(W_k)$, $1\leq m\leq\dim V$, $1\leq n\leq\dim W_k$ be random variables satisfying conditions of Theorem~\ref{th:3}. Consider random field \eqref{eq:spectralreducible}. Its mean value is obviously constant. Its correlation operator is
\[
\begin{aligned}
R^{(kk')}_{jj'}(t_1,t_2)&=\sum_{V\in\hat{G}_K(W_k)\cap\hat{G}_K(W_{k'})}
\sum^{\dim W_k}_{n=1}\sum^{\dim W_{k'}}_{n'=1}R^{(V)}_{kn,k'n'}\sum^{\dim V}_{m=1}
({}_{W_k}Y_{Vm})_j(t_1)\overline{({}_{W_{k'}}Y_{Vm})_{j'}(t_2)}\\
&=\frac{1}{\sqrt{\dim W_k\dim W_{k'}}}\sum_{V\in\hat{G}_K(W_k)\cap\hat{G}_K(W_{k'})}\dim V\sum^{\dim W_k}_{n=1}\sum^{\dim W_{k'}}_{n'=1}R^{(V)}_{kn,k'n'}\\
&\quad\times V_{p+j,p+j'}(g^{-1}_1g_2),
\end{aligned}
\]
with the same notation as in proof of Theorem~\ref{th:2}. The uniform convergence of the above series and the invariance of the field \eqref{eq:spectralreducible} is proved exactly in the same way as in proof of Theorem~\ref{th:2}.
\end{proof}

\section{Application to cosmology}\label{sec:cosmology}

\subsection{The cosmic microwave background}\label{sub:deterministic}

Let $\mathbf{E}(\mathbf{n})\in T_{\mathbf{n}}S^2$ be the electric field of the cosmic microwave background. From the observations, we define the intensity tensor. In physical terms, the intensity tensor is
\[
\mathcal{P}=C\langle\mathbf{E}(\mathbf{n})\otimes\mathbf{E}^*(\mathbf{n})\rangle,
\]
where $\langle\boldsymbol\cdot\rangle$ denote time average over the historical accidents that produced a particular pattern of fluctuations. Assuming ergodicity, time average is equal to the space average, i.e., average over the possible positions from which the radiation could be observed. The constant $C$ is chosen so that $\mathcal{P}$ is measured in brightness temperature units (in these units, the intensity tensor is independent of radiation frequency). It will be ignored in what follows.

Introduce a basis in each tangent plane $T_{\mathbf{n}}S^2$. Realise $S^2$ as $\{\,(x,y,z)\in\mathbb{R}^3\colon x^2+y^2+z^2=1\,\}$ and define the chart $(U_I,\mathbf{h}_I)$ as $U_I=S^2\setminus\{(0,0,1),(0,0,-1)\}$ and $\mathbf{h}_I(\mathbf{n})=(\theta(\mathbf{n}),\varphi(\mathbf{n}))\in\mathbb{R}^2$, the spherical coordinates. Let $SO(3)$ be the rotation group in $\mathbb{R}^3$. For any rotation $g$, define the chart $(U_g,\mathbf{h}_g)$ as
\[
U_g=gU_I,\qquad\mathbf{h}_g(\mathbf{n})=h_I(g^{-1}\mathbf{n}).
\]
The sphere $S^2$, equipped with the atlas $\{\,(U_g,\mathbf{h}_g)\colon g\in SO(3)\,\}$, becomes the real-analytic manifold. The local $\theta$-axis in each tangent plane is along the direction of decreasing the inclination $\theta$:
\[
\mathbf{e}_{\theta}=-\frac{\partial}{\partial\theta}.
\]
The local $\varphi$-axis is along the direction of increasing the azimuth $\varphi$:
\[
\mathbf{e}_{\varphi}=(1/\sin\theta)\frac{\partial}{\partial\varphi}.
\]
With this convention, $\mathbf{e}_{\theta}$, $\mathbf{e}_{\varphi}$, and the direction of radiation propagation $-\mathbf{n}$ form a right-handed basis. This convention is in accordance with the International Astronomic Union standard. The orthonormal  basis $(\mathbf{e}_{\theta},\mathbf{e}_{\varphi})$ turns $S^2$ into a Riemannian manifold and each tangent plane $T_{\mathbf{n}}S^2$ can be identified with the space $\mathbb{R}^2$.

In the just introduced basis, the intensity tensor becomes the intensity matrix:
\[
\mathcal{P}_{ab}=\langle\mathbf{E}_a(\mathbf{n})\otimes\mathbf{E}^*_b(\mathbf{n})\rangle,
\qquad a,b\in\{\theta,\varphi\}.
\]
The rotations about the line of sight together with \emph{parity transformation} $\mathbf{n}\to-\mathbf{n}$ generate the group $O(2)$ of orthogonal matrices in $\mathbb{R}^2$. The action of $O(2)$ on the intensity matrix extends to the representation $g\mapsto g\mathcal{A}g^{-1}$ of $O(2)$ in the real $4$-dimensional space of Hermitian $2\times 2$ matrices $\mathcal{A}$ with inner product
\[
(\mathcal{A},\mathcal{B})=\tr(\mathcal{A}\mathcal{B}).
\]
This representation is reducible and may be decomposed into the direct sum of three irreducible representations.

The standard choice of an orthonormal basis in the spaces of the irreducible components is as follows. The space of the first irreducible component is generated by the matrix
\[
\frac{1}{2}\sigma_0=\frac{1}{2}
\begin{pmatrix}
1 & 0\\
0 & 1
\end{pmatrix}
\]
The representation in this space is the trivial representation of the group $O(2)$. Physicists call the elements of this space \emph{scalars}.

The space of the second irreducible component is generated by the matrices
\[
\frac{1}{2}\sigma_1=\frac{1}{2}
\begin{pmatrix}
0 & 1\\
1 & 0
\end{pmatrix}
,\qquad\frac{1}{2}\sigma_3=\frac{1}{2}
\begin{pmatrix}
1 & 0\\
0 & -1
\end{pmatrix}
.
\]
Let $g_{\alpha}\in SO(2)$ with
\begin{equation}\label{eq:so2}
g_{\alpha}=
\begin{pmatrix}
\cos\alpha & \sin\alpha\\
-\sin\alpha & \cos\alpha
\end{pmatrix}
.
\end{equation}
It is easy to check that
\begin{equation}\label{eq:rotation}
\begin{aligned}
g_{\alpha}\sigma_1g^{-1}_{\alpha}&=\cos(2\alpha)\sigma_1+\sin(2\alpha)\sigma_3,\\
g_{\alpha}\sigma_3g^{-1}_{\alpha}&=-\sin(2\alpha)\sigma_1+\cos(2\alpha)\sigma_3.
\end{aligned}
\end{equation}
The elements of this space are \emph{symmetric trace-free tensors}.

Finally, the space of the third irreducible component is generated by the matrix
\[
\frac{1}{2}\sigma_2=\frac{1}{2}
\begin{pmatrix}
0 & -\mathrm{i}\\
\mathrm{i} & 0
\end{pmatrix}
\]
The representation in this space is the representation $g\mapsto\det g$ of the group $O(2)$. Physicists call the elements of this space \emph{pseudo-scalars} (they do not change under rotation but change sign under reflection). The matrices $\sigma_1$, $\sigma_2$, and $\sigma_3$ are known as \emph{Pauli matrices}.

The standard physical notation for the components of the intensity matrix in the above basis is as follows:
\[
\mathcal{P}=\frac{1}{2}(I\sigma_0+U\sigma_1+V\sigma_2+Q\sigma_3),
\]
or
\[
\mathcal{P}=\frac{1}{2}
\begin{pmatrix}
I+Q & U-\mathrm{i}V\\
U+\mathrm{i}V & I-Q
\end{pmatrix}
.
\]
The real numbers $I$, $Q$, $U$, and $V$ are called \emph{Stokes parameters}. Their physical sense is as follows. $I$ is the total intensity of the radiation (which is directly proportional to the fourth power of the absolute temperature $T$ by the Stefan--Boltzmann law). On the tangent plane $T_{\mathbf{n}}S^2$, the tip of the electric vector $\mathbf{E}(\mathbf{n})$ traces out an ellipse as a function of time. The parameters $U$ and $Q$ measure the orientation of the above ellipse relative to the local $\theta$-axis, $\mathbf{e}_{\theta}$. The \emph{polarisation angle} between the major axis of the ellipse and $\mathbf{e}_{\theta}$ is
\[
\chi=\frac{1}{2}\tan^{-1}\frac{U}{Q},
\]
and the length of the major semi-axis is $(Q^2+U^2)^{1/2}$. The last parameter, $V$, measures circular polarisation.

According to modern cosmological theories, the polarisation of the CMB was introduced while scattering off the photons by charged particles. This process cannot induce circular polarisation in the scattered light. Therefore, in what follows we put $V=0$.

The physics of the CMB polarisation is described in Cabella and Kamionkowski (2005), Challinor (2004), Challinor (2009), Challinor and Peiris (2009), Durrer (2008), Lin and Wandelt (2006), among others. Of these, Challinor and Peiris use the right-hand basis, while the remaining authors use the left-hand basis, in which $\mathbf{e}_{\theta}=\partial/\partial\theta$. In what follows,  we use the left-hand basis $\mathbf{e}_{\theta}$, $\mathbf{e}_{\varphi}$, $-\mathbf{n}$ with
\begin{equation}\label{eq:realbasis}
\mathbf{e}_{\theta}=\frac{\partial}{\partial\theta},\qquad
\mathbf{e}_{\varphi}=(1/\sin\theta)\frac{\partial}{\partial\varphi}.
\end{equation}

\subsection{The probabilistic model of the CMB}\label{sub:probabilistic}

The absolute temperature, $T(\mathbf{n})$, is a section of the homogeneous vector bundle $\xi_0=(\mathcal{E}_0,\pi,S^2)$, where the representation of the rotation group $G=SO(3)$ induced by the representation $W(g_{\alpha})=1$ of the massive subgroup $K=SO(2)$ is realised.

The representations $V$ of the group $G$ are enumerated by nonnegative integers $\ell=0$, $1$, \dots. The restriction of the representation $V_{\ell}$ onto $K$ is the direct sum of the representations $e^{\mathrm{i}m\alpha}$, $m=-\ell$, $-\ell+1$, \dots, $\ell$. Therefore we have $\dim V_{\ell}=2\ell+1$ and $|m|\leq\ell$. By Frobenius reciprocity,  $\hat{G}_K(W)=\{V_0,V_1,\dots,V_{\ell},\dots\}$.

The representations $e^{\mathrm{i}m\alpha}$ of $K$ act in one-dimensional complex spaces $H_m$. To define a basis in the space $H^{(\ell)}$ of the representation $V_{\ell}$, choose a unit vector $\mathbf{e}_m$ in each space $H_m$. Each vector $\mathbf{e}_m$ of a basis can be multiplied by a \emph{phase} $e^{\mathrm{i}\alpha_m}$. The choice of a phase is called the \emph{phase convention}.

Any rotation $g\in SO(3)$ is defined by the Euler angles $g=(\varphi,\theta,\psi)$ with $\varphi$, $\psi\in[0,2\pi]$ and $\theta\in[0,\pi]$. The order in which the angles are given and the axes about which they are applied are not subject of a standard. We adopt the so called \emph{$zxz$ convention}: the first rotation is about the $z$-axis by $\psi$, the second rotation is about the $x$-axis by $\theta$, and the third rotation is about $z$-axis by $\varphi$. Note that the chart defined by the Euler angles satisfies our condition: the first two local coordinates $(\varphi,\theta)$ are spherical coordinates in $S^2$ (up to order) with dense domain $U_I$.

The matrix elements of the representation $V_{\ell}$ are traditionally denoted by
\[
D^{(\ell)}_{mn}(\varphi,\theta,\psi)=(V_{\ell}(\varphi,\theta,\psi)
\mathbf{e}_m,\mathbf{e}_n)_{H^{(\ell)}}
\]
and called \emph{Wigner $D$-functions}. The explicit formula for the Wigner $D$-function depends on the phase convention. Choose the basis $\{\,\mathbf{e}_m\colon-\ell\leq m\leq\ell\,\}$ in every space $H^{(\ell)}$ to obtain
\[
D^{(\ell)}_{mn}(\varphi,\theta,\psi)=e^{-\mathrm{i}m\varphi}
d^{(\ell)}_{mn}(\theta)e^{-\mathrm{i}n\psi},
\]
where $d^{(\ell)}_{mn}(\theta)$ are \emph{Wigner $d$-functions}:
\begin{equation}\label{eq:durrerelements}
\begin{aligned}
d^{(\ell)}_{mn}(\theta)&=(-1)^m\sqrt{\frac{(\ell+m)!(\ell-m)!}{(\ell+n)!(\ell-n)!}}
\sin^{2\ell}(\theta/2)\\
&\quad\times\sum^{\min\{\ell+m,\ell+n\}}_{r=\max\{0,m+n\}}\binom{\ell+n}{r}
\binom{\ell-n}{r-m-n}(-1)^{\ell-r+n}\cot^{2r-m-n}(\theta/2).
\end{aligned}
\end{equation}
The following symmetry relation follows.
\begin{equation}\label{eq:symmetry}
d^{(\ell)}_{-m,-n}(\theta)=(-1)^{n-m}d^{(\ell)}_{mn}(\theta)
\end{equation}

In this particular case, formula \eqref{eq:spherical} takes the form
\[
{}_WY_{\ell m}(\theta,\varphi)=\sqrt{2\ell+1}\overline{D^{(\ell)}_{m0}(\varphi,\theta,0)}.
\]
The functions in the left hand side form an orthonormal basis in the space of the square integrable functions on $S^2$ with respect to the probabilistic $SO(3)$-invariant measure. It is conventional to form a basis with respect to the Lebesgue measure induced by the embedding $S^2\subset\mathbb{R}^3$ which is $4\pi$ times the probabilistic invariant measure, and omit the first subscript:
\[
Y_{\ell m}(\theta,\varphi)=\sqrt{\frac{2\ell+1}{4\pi}}\overline{D^{(\ell)}_{m0}(\varphi,\theta,0)}.
\]
This is formula (A4.40) from Durrer (2008) defining the \emph{spherical harmonics}.

In cosmological models, one assumes that $T(\mathbf{n})$ is a single realisation of the mean square continuous strict sense $SO(3)$-invariant random field in the homogeneous vector bundle $\xi=(\mathcal{E}_0,\pi,S^2)$. It is custom to use the term ``\emph{isotropic}" instead of ``$SO(3)$-invariant". By Theorem~\ref{th:2}, we have
\[
T(\mathbf{n})=\sum^{\infty}_{\ell=0}\sum^{\ell}_{m=-\ell}Z_{\ell m}Y_{\ell m}(\mathbf{n}),
\]
where $\mathsf{E}[Z_{\ell m}]=0$ unless $\ell=0$ and $\mathsf{E}[Z_{\ell m}\overline{Z_{\ell'm'}}]=\delta_{\ell\ell'}\delta_{mm'}R^{(\ell)}$ with
\[
\sum^{\infty}_{\ell=0}(2\ell+1)R^{(\ell)}<\infty.
\]
This formula goes back to Obukhov (1947).

Physicists call $Z_{\ell m}$s the \emph{expansion coefficients}, and $R^{(\ell)}$ the \emph{power spectrum} of the CMB. Different notations for the expansion coefficients and power spectrum may be found in the literature. Some of them are shown in Table~\ref{tab1}.

\begin{table}
\centering
\begin{tabular}{|l|c|c|}
\hline \textbf{Source} & $Z_{\ell m}$ & $R^{(\ell)}$ \\
\hline Cabella and Kamionkowski (2005) & $a^T_{\ell m}$ & $C^{TT}_{\ell}$ \\
\hline Challinor (2005), & & \\
Challinor and Peiris (2009) & $T_{\ell m}$ & $C^T_{\ell}$ \\
\hline Durrer (2008), & & \\
Weinberg (2008) & $a_{\ell m}$ & $C_{\ell}$ \\
\hline Lin and Wandelt (2006), & & \\
Zaldarriaga and Seljak (1997) & $a_{T,\ell m}$ & $C_{T\ell}$ \\
\hline Kamionkowski et al (1997) & $a^T_{\ell m}$ & $C^T_{\ell}$ \\
\hline
\end{tabular}
\caption{Examples of different notation for temperature expansion coefficients and power spectrum.}\label{tab1}
\end{table}

In what follows, we use notation of Lin ans Wandelt (2006). In this notation, the expansion for the temperature has the form
\begin{equation}\label{eq:temperature}
T(\mathbf{n})=\sum^{\infty}_{\ell=0}\sum^{\ell}_{m=-\ell}a_{T,\ell m}Y_{\ell m}(\mathbf{n}).
\end{equation}
Since $T(\mathbf{n})$ is real, the coefficients $a_{T,\ell m}$ must satisfy the \emph{reality condition} which depends on the phase convention. For our current convention, when the Wigner $d$-function is determined by \eqref{eq:durrerelements}, we have
\[
\begin{aligned}
Y_{\ell\;-m}(\theta,\varphi)&=\sqrt{\frac{2\ell+1}{4\pi}}
e^{-\mathrm{i}m\varphi}d^{(\ell)}_{-m,0}(\theta)\\
&=\sqrt{\frac{2\ell+1}{4\pi}}
e^{-\mathrm{i}m\varphi}(-1)^md^{(\ell)}_{m0}(\theta)\\
&=(-1)^m\overline{Y_{\ell m}(\theta,\varphi)}.
\end{aligned}
\]
Here we used the symmetry relation \eqref{eq:symmetry}. The reality condition is
\begin{equation}\label{eq:ordinaryreality}
a_{T,\ell\;-m}=(-1)^m\overline{a_{T,\ell m}}.
\end{equation}
This form of the reality condition is used by Cabella and Kamionkowski (2005), Challinor (2004, 2009), Challinor and Peiris (2009), Durrer (2008), Kamionkowski et al (1997), Lin and Wandelt (2006), among others.

Introduce the following notation: $m^+=\max\{m,0\}=(|m|+m)/2$, $m^-=\max\{-m,0\}=(|m|-m)/2$. We have $(-m)^-=m^+$ and $m^+-m=m^-$. If we choose another basis, $\{\,(-1)^{m^-}\mathbf{e}_m\colon-\ell\leq m\leq\ell\,\}$, then the Wigner $D$-function, $D^{(\ell)}_{mn}(\varphi,\theta,\psi)$, is multiplying by $(-1)^{m^-+n^-}$, and we obtain
\[
\begin{aligned}
Y_{\ell\;-m}(\theta,\varphi)&=\sqrt{\frac{2\ell+1}{4\pi}}
(-1)^{(-m)^-}e^{-\mathrm{i}m\varphi}d^{(\ell)}_{-m,0}(\theta)\\
&=\sqrt{\frac{2\ell+1}{4\pi}}
(-1)^{m+m^-}e^{-\mathrm{i}m\varphi}(-1)^md^{(\ell)}_{m0}(\theta)\\
&=\overline{Y_{\ell m}(\theta,\varphi)}.
\end{aligned}
\]
The modified reality condition is
\begin{equation}\label{eq:modifiedreality}
a_{T,\ell\;-m}=\overline{a_{T,\ell m}}.
\end{equation}
This form of reality condition is used by Geller and Marinucci (2008), Weinberg (2008), Zaldarriaga and Seljak (1997), among others.

Let $T_0=\mathsf{E}[T(\mathbf{n})]$. The temperature fluctuation, $\Delta T(\mathbf{n})=T(\mathbf{n})-T_0$, expands as
\[
\Delta T(\mathbf{n})=\sum^{\infty}_{\ell=1}\sum^{\ell}_{m=-\ell}a_{T,\ell m}Y_{\ell m}(\mathbf{n}).
\]
The part of this sum corresponding to $\ell=1$ is called a \emph{dipole}. When analysing data, the dipole is usually removed since it linearly depends on the velocity of the observer's motion relative to the surface of last scattering.

The \emph{complex polarisation} is defined as $Q+\mathrm{i}U$. It follows easily from \eqref{eq:rotation} that any rotation \eqref{eq:so2} maps $Q+\mathrm{i}U$ to $e^{2\mathrm{i}\alpha}(Q+\mathrm{i}U)$. Then, by \eqref{eq:localrotation}, $(Q+\mathrm{i}U)(\mathbf{n})$ is a section of the homogeneous vector bundle $\xi_{-2}=(\mathcal{E}_{-2},\pi,S^2)$, where the representation of the rotation group $G=SO(3)$ induced by the representation $W(g_{\alpha})=e^{-2\mathrm{i}\alpha}$ of the massive subgroup $K=SO(2)$ is realised. By Frobenius reciprocity,  $\hat{G}_K(W)=\{V_2,V_3,\dots,V_{\ell},\dots\}$.

In general, let $s\in\mathbb{Z}$, and let $\xi_{-s}=(\mathcal{E}_{-s},\pi,S^2)$ be the homogeneous vector bundle where the representation of the rotation group $SO(3)$ induced by the representation $W(g_{\alpha})=e^{-\mathrm{i}s\alpha}$ of the massive subgroup $SO(2)$ is realised. In the physical literature, the sections of these bundle are called

\begin{itemize}

\item quantities of spin $s$ by Challinor (2009), Challinor and Peiris (2009), Geller and Marinucci (2008), Newman and Penrose (1966), Weinberg (2008) among others;

\item quantities of spin $-s$ by Cabella and Kamionkowski (2005), Lin and Wendelt (2006), Zaldarriaga and Seljak (1997) among others;

\item quantities of spin $|s|$ and helicity $s$ by Durrer (2008) among others.

\end{itemize}

Let $g=(\theta,\varphi,\psi)$ be the Euler angles in $SO(3)$. Put $\psi=0$. Then, $\mathbf{n}=(\theta,\varphi,0)$ are spherical coordinates in $S^2$. By \eqref{eq:deterministic} we obtain
\[
(Q+\mathrm{i}U)(\mathbf{n})=\sum^{\infty}_{\ell=2}\sum^{\ell}_{m=-\ell}
a_{-2,\ell m}\;{}_{-2}Y_{\ell m}(\mathbf{n}),
\]
where
\[
a_{-2,\ell m}=\int_{S^2}(Q+\mathrm{i}U)(\mathbf{n})
\overline{{}_{-2}Y_{\ell m}(\mathbf{n})}\,\mathrm{d}\mathbf{n}
\]
and, by \eqref{eq:spherical},
\[
{}_{-2}Y_{\ell m}(\theta,\varphi)=\sqrt{2\ell+1}\overline{D^{(\ell)}_{m,-2}(\varphi,\theta,0)}.
\]
The functions in the left hand side form an orthonormal basis in the space of the square integrable sections of the homogeneous vector bundle $\xi_{-2}$ with respect to the probabilistic $SO(3)$-invariant measure.

There exist different conventions. The first convention is used by Durrer (2008) among others. In this convention, a basis is formed with respect to the Lebesgue measure induced by the embedding $S^2\subset\mathbb{R}^3$ which is $4\pi$ times the probabilistic invariant measure and the sign of the second index of the Wigner $D$-function is changed (because we would like to expand $Q+\mathrm{i}U$ with respect to ${}_2Y_{\ell m}$):
\[
{}_{-2}Y_{\ell m}(\theta,\varphi)=\sqrt{\frac{2\ell+1}{4\pi}}\overline{D^{(\ell)}_{m,2}(\varphi,\theta,0)}.
\]

In the general case, for any $s\in\mathbb{Z}$, this convention reads (Durrer (2008), formula (A4.51))
\begin{equation}\label{eq:durrerharmonics}
{}_sY_{\ell m}(\theta,\varphi)=\sqrt{\frac{2\ell+1}{4\pi}}\overline{D^{(\ell)}_{m,-s}(\varphi,\theta,0)}.
\end{equation}
These functions are called \emph{spherical harmonics of spin~$s$} or the \emph{spin-weighted spherical harmonics}. They appeared in Gelfand and Shapiro (1952) under the name \emph{generalised spherical harmonics}. The current name goes back to Newman and Penrose (1966). Note that the spin-weighted spherical harmonics are defined for $\ell\geq|s|$ and $|m|\leq\ell$.

The second harmonic convention is used by Lin and Wandelt (2006), Newman and Penrose (1966), among others. It reads as
\[
{}_sY_{\ell m}(\theta,\varphi)=(-1)^m\sqrt{\frac{2\ell+1}{4\pi}}\overline{D^{(\ell)}_{m,-s}(\varphi,\theta,0)}.
\]

Both conventions are coherent with the following phase convention:
\begin{equation}\label{eq:phase}
\overline{{}_sY_{\ell m}}=(-1)^{m+s}{}_{-s}Y_{\ell\;-m}.
\end{equation}
In particular, for $s=0$ we return back to the convention $\overline{Y_{\ell m}}=(-1)^mY_{\ell\;-m}$ corresponding to reality condition \eqref{eq:ordinaryreality}.

To produce the harmonic convention coherent with the phase convention
\[
\overline{{}_sY_{\ell m}}=(-1)^s{}_{-s}Y_{\ell\;-m}
\]
corresponding to reality condition \eqref{eq:modifiedreality}, one must multiply the right hand side of the convention equation by $(-1)^{m^-}$. Thus, the modified first convention, used by Weinberg (2008) among others, is
\[
{}_sY_{\ell m}(\theta,\varphi)=(-1)^{m^-}\sqrt{\frac{2\ell+1}{4\pi}}\overline{D^{(\ell)}_{m,-s}(\varphi,\theta,0)},
\]
while the modified second convention, used by Geller and Marinucci (2008), among others, is
\[
{}_sY_{\ell m}(\theta,\varphi)=(-1)^{m^+}\sqrt{\frac{2\ell+1}{4\pi}}\overline{D^{(\ell)}_{m,-s}(\varphi,\theta,0)}.
\]

In what follows, we use the convention \eqref{eq:durrerharmonics}. The explicit expression for the spherical harmonics of spin~$s$ in the chart determined by spherical coordinates follows from \eqref{eq:durrerelements} and \eqref{eq:durrerharmonics}:
\begin{equation}\label{eq:explicit}
\begin{aligned}
{}_sY_{\ell m}(\theta,\varphi)&=(-1)^m\sqrt{\frac{(2\ell+1)(\ell+m)!(\ell-m)!}{4\pi(\ell+s)!(\ell-s)!}}
\sin^{2\ell}(\theta/2)e^{\mathrm{i}m\varphi}\\
&\quad\times\sum^{\min\{\ell+m,\ell-s\}}_{r=\max\{0,m-s\}}\binom{\ell-s}{r}
\binom{\ell+s}{r-m+s}(-1)^{\ell-r-s}\cot^{2r-m+s}(\theta/2).
\end{aligned}
\end{equation}
The decomposition of the complex polarisation takes the form
\[
(Q+\mathrm{i}U)(\mathbf{n})=\sum^{\infty}_{\ell=2}\sum^{\ell}_{m=-\ell}
a_{2,\ell m}\;{}_2Y_{\ell m}(\mathbf{n}),
\]
where
\[
a_{2,\ell m}=\int_{S^2}(Q+\mathrm{i}U)(\mathbf{n})
\overline{{}_2Y_{\ell m}(\mathbf{n})}\,\mathrm{d}\mathbf{n},
\]
while the decomposition of the conjugate complex polarisation is
\[
(Q-\mathrm{i}U)(\mathbf{n})=\sum^{\infty}_{\ell=2}\sum^{\ell}_{m=-\ell}
a_{-2,\ell m}\;{}_{-2}Y_{\ell m}(\mathbf{n}),
\]
where
\[
a_{-2,\ell m}=\int_{S^2}(Q-\mathrm{i}U)(\mathbf{n})
\overline{{}_{-2}Y_{\ell m}(\mathbf{n})}\,\mathrm{d}\mathbf{n}.
\]

In cosmological models, one assumes that $(Q+iU)(\mathbf{n})$ is a single realisation of the mean square continuous strict sense isotropic random field in the homogeneous vector bundle $\xi_2$. Isotropic random fields in vector bundles $\xi_s$, $s\in\mathbb{Z}$ were defined by Geller and Marinucci (2008). By Theorem~\ref{th:2}, we have
\begin{equation}\label{eq:polarisation}
(Q+\mathrm{i}U)(\mathbf{n})=\sum^{\infty}_{\ell=2}\sum^{\ell}_{m=-\ell}a_{2,\ell m}\;{}_2Y_{\ell m}(\mathbf{n}),
\end{equation}
where $\mathsf{E}[a_{2,\ell m}]=0$ and $\mathsf{E}[a_{2,\ell m}\overline{a_{2,\ell'm'}}]=\delta_{\ell\ell'}\delta_{mm'}C_{2\ell}$ with
\[
\sum^{\infty}_{\ell=2}(2\ell+1)C_{2\ell}<\infty.
\]

Different notations for the complex polarisation expansion coefficients $a_{\pm 2,\ell m}$ may be found in the literature. Some of them are shown in Table~\ref{tab2}.

\begin{table}
\centering
\begin{tabular}{|l|c|}
\hline \textbf{Source} & \textbf{Expansion coefficients} \\
\hline Durrer (2008) & $a^{(\pm 2)}_{\ell m}$ \\
\hline Lin and Wandelt (2006), & \\
Zaldarriaga and Seljak (1997) & $a_{\pm 2,\ell m}$ \\
\hline Weinberg (2008) & $a_{P,\ell m}$ \\
\hline
\end{tabular}
\caption{Examples of different notation for complex polarisation expansion coefficients.}\label{tab2}
\end{table}

In what follows we use the notation by Lin and Wandelt (2006). The expansion for the conjugate complex polarisation has the form
\begin{equation}\label{eq:conjugatepolarisation}
(Q-\mathrm{i}U)(\mathbf{n})=\sum^{\infty}_{\ell=2}\sum^{\ell}_{m=-\ell}a_{-2,\ell m}\;{}_{-2}Y_{\ell m}(\mathbf{n}).
\end{equation}

Since $Q(\mathbf{n})$ and $U(\mathbf{n})$ are real, the coefficients $a_{2,\ell m}$ and $a_{-2,\ell m}$ must satisfy the reality condition which depends on the phase convention. We agreed to use the first harmonic convention \eqref{eq:durrerharmonics}. Therefore, our phase convention is \eqref{eq:phase}, and the reality condition is
\begin{equation}\label{eq:spinreality}
\overline{a_{-2,lm}}=(-1)^ma_{2,l\;-m}.
\end{equation}

Along with the standard basis \eqref{eq:realbasis}, it is useful to use the so called \emph{helicity basis}. Again, there exist different names and conventions. Durrer (2008) defines the helicity basis as
\[
\mathbf{e}_+=\frac{1}{\sqrt{2}}(\mathbf{e}_{\theta}-\mathrm{i}\mathbf{e}_{\varphi}),\qquad
\mathbf{e}_-=\frac{1}{\sqrt{2}}(\mathbf{e}_{\theta}+\mathrm{i}\mathbf{e}_{\varphi}),
\]
while Weinberg (2008) uses the opposite definition
\[
\mathbf{e}_+=\frac{1}{\sqrt{2}}(\mathbf{e}_{\theta}+\mathrm{i}\mathbf{e}_{\varphi}),\qquad
\mathbf{e}_-=\frac{1}{\sqrt{2}}(\mathbf{e}_{\theta}-\mathrm{i}\mathbf{e}_{\varphi}).
\]
Challinor (2005) and Thorne (1980) use notation
\[
\mathbf{m}=\frac{1}{\sqrt{2}}(\mathbf{e}_{\theta}+\mathrm{i}\mathbf{e}_{\varphi}),\qquad
\mathbf{m}^*=\frac{1}{\sqrt{2}}(\mathbf{e}_{\theta}-\mathrm{i}\mathbf{e}_{\varphi}),
\]
while Challinor and Peiris (2009) use notation
\[
\mathbf{m}_+=\frac{1}{\sqrt{2}}(\mathbf{e}_{\theta}+\mathrm{i}\mathbf{e}_{\varphi}),\qquad
\mathbf{m}_-=\frac{1}{\sqrt{2}}(\mathbf{e}_{\theta}-\mathrm{i}\mathbf{e}_{\varphi})
\]
and call these the \emph{null basis}. We will use the definition and notation by Durrer (2008).

The helicity basis is useful by the following reason. Let $\eth$ be a covariant derivative in direction $-\sqrt{2}\mathbf{e}_-$:
\[
\eth=\nabla_{-\sqrt{2}\mathbf{e}_-}.
\]
Let $C^{\infty}(\xi_s)$ be the space of infinitely differentiable sections of the vector bundle $\xi_s$.
Durrer (2008) proves that for any ${}_sf\in C^{\infty}(\xi_s)$ we have
\[
\eth\;{}_sf=\left(s\cot\theta-\frac{\partial}{\partial\theta}-\frac{\mathrm{i}}{\sin\theta}
\frac{\partial}{\partial\varphi}\right){}_sf.
\]
In particular, put ${}_sf={}_sY_{\ell m}$. Using \eqref{eq:explicit}, we obtain
\[
\eth\;{}_sY_{\ell m}=\sqrt{(\ell-s)(\ell+s+1)}{}_{s+1}Y_{\ell m}.
\]
For $s\geq 0$ and $\ell=s$, the spherical harmonic ${}_{s+1}Y_{\ell m}$ is not defined and we use convention $\sqrt{(\ell-\ell)(2\ell+1)}{}_{\ell+1}Y_{\ell m}=0$. Then, $\eth\colon C^{\infty}(\xi_s)\to C^{\infty}(\xi_{s+1})$. Therefore, $\eth$ is called the \emph{spin raising operator}. Moreover, the last display shows that the restriction of $\eth$ onto the space $H^{(\ell)}$, $\ell>s$, is an intertwining  operator between equivalent representations $V_{\ell}$.

The adjoint operator, $\eth^*$, is a covariant derivative in direction $-\sqrt{2}\mathbf{e}_+$:
\[
\eth^*=\nabla_{-\sqrt{2}\mathbf{e}_+}.
\]
For any ${}_sf\in C^{\infty}(\xi_s)$ we have
\[
\eth^*{}_sf=\left(s\cot\theta-\frac{\partial}{\partial\theta}+\frac{\mathrm{i}}{\sin\theta}
\frac{\partial}{\partial\varphi}\right){}_sf.
\]
In particular,
\[
\eth^*{}_sY_{\ell m}=-\sqrt{(\ell+s)(\ell-s+1)}{}_{s-1}Y_{\ell m}.
\]
For $s\leq 0$ and $\ell=-s$, the spherical harmonic ${}_{s-1}Y_{\ell m}$ is not defined and we use convention $\sqrt{(\ell-\ell)(2\ell+1)}{}_{-\ell-1}Y_{\ell m}=0$. Then, $\eth^*\colon C^{\infty}(\xi_s)\to C^{\infty}(\xi_{s-1})$. Therefore, $\eth^*$ is called the \emph{spin lowering operator}. Moreover, the last display shows that the restriction of $\eth^*$ onto the space $H^{(\ell)}$, $\ell>-s$, is an intertwining  operator between equivalent representations $V_{\ell}$.

Zaldarriaga and Seljak (1997) introduced the following idea. Assume for a moment that
\begin{equation}\label{eq:assumption}
\sum^{\infty}_{\ell=2}\frac{(2\ell+1)(\ell+2)!}{(l-2)!}C_{2\ell}<\infty.
\end{equation}
Then, it is possible to act twice with $\eth$ on both hand sides of \eqref{eq:conjugatepolarisation} and to interchange differentiation and summation:
\[
\begin{aligned}
\eth^2(Q-\mathrm{i}U)(\mathbf{n})&=\eth^2\sum^{\infty}_{\ell=2}\sum^{\ell}_{m=-\ell}a_{-2,\ell m}\;{}_{-2}Y_{\ell m}(\mathbf{n})\\
&=\sum^{\infty}_{\ell=2}\sum^{\ell}_{m=-\ell}a_{-2,\ell m}\eth^2{}_{-2}Y_{\ell m}(\mathbf{n})\\
&=\sum^{\infty}_{\ell=2}\sum^{\ell}_{m=-\ell}\sqrt{\frac{(\ell+2)!}{(\ell-2)!}}a_{-2,\ell m}Y_{\ell m}(\mathbf{n}).
\end{aligned}
\]
By the same argumentation, we have
\[
(\eth^*)^2(Q+\mathrm{i}U)(\mathbf{n})=\sum^{\infty}_{\ell=2}\sum^{\ell}_{m=-\ell}\sqrt{\frac{(\ell+2)!}{(\ell-2)!}}a_{2,\ell m}Y_{\ell m}(\mathbf{n}).
\]
Unlike complex polarisation, the new random fields are rotationally invariant and no ambiguities connected with  rotations \eqref{eq:rotation} arise. However, they have complex behaviour under parity transformation, because $Q(\mathbf{n})$ and $U(\mathbf{n})$ behave differently (Lin and Wandelt (2006)): $Q$ has even parity: $Q(-\mathbf{n})=Q(\mathbf{n})$ while $U$ has odd parity: $U(-\mathbf{n})=-U(\mathbf{n})$.

Therefore, it is custom to group together quantities of the same parity:
\[
\begin{aligned}
\tilde{E}(\mathbf{n})&=-\frac{1}{2}((\eth^*)^2(Q+\mathrm{i}U)(\mathbf{n})+\eth^2(Q-\mathrm{i}U)(\mathbf{n})),\\
\tilde{B}(\mathbf{n})&=-\frac{1}{2\mathrm{i}}((\eth^*)^2(Q+\mathrm{i}U)(\mathbf{n})-\eth^2(Q-\mathrm{i}U)(\mathbf{n})).
\end{aligned}
\]
The random fields $\tilde{E}(\mathbf{n})$ and $\tilde{B}(\mathbf{n})$ are scalar (spin $0$), real-valued, and isotropic. To find their behaviour under parity transformation, follow Lin and Wandelt (2006). Notice that if $\mathbf{n}$ has spherical coordinates $(\theta,\varphi)$, then $-\mathbf{n}$ has spherical coordinates $\theta'=\pi-\theta$ and $\varphi'=\varphi+\pi$. Therefore,
\[
\frac{\partial}{\partial\theta'}=-\frac{\partial}{\partial\theta},\qquad
\frac{\partial}{\partial\varphi'}=\frac{\partial}{\partial\varphi}.
\]
Because $(Q+\mathrm{i}U)(-\mathbf{n})=(Q-\mathrm{i}U)(\mathbf{n})$, we obtain
\[
\begin{aligned}
(\eth^*)'(Q+\mathrm{i}U)(-\mathbf{n})&=\left(2\cot\theta'-\frac{\partial}{\partial\theta'}+\frac{\mathrm{i}}{\sin\theta'}
\frac{\partial}{\partial\varphi'}\right)(Q+\mathrm{i}U)(-\mathbf{n})\\
&=\left(-2\cot\theta+\frac{\partial}{\partial\theta}-\frac{\mathrm{i}}{\sin\theta}
\frac{\partial}{\partial\varphi}\right)(Q-\mathrm{i}U)(\mathbf{n})\\
&=-\eth(Q-\mathrm{i}U)(\mathbf{n})
\end{aligned}
\]
and
\[
\begin{aligned}
((\eth^*)')^2(Q+\mathrm{i}U)(-\mathbf{n})&=\left(2\cot\theta'-\frac{\partial}{\partial\theta'}+\frac{\mathrm{i}}{\sin\theta'}
\frac{\partial}{\partial\varphi'}\right)(-\eth(Q-\mathrm{i}U)(\mathbf{n}))\\
&=\eth^2(Q-\mathrm{i}U)(\mathbf{n}).
\end{aligned}
\]
Similarly, we have $(\eth')^2(Q+\mathrm{i}U)(-\mathbf{n})=(\eth^*)^2(Q-\mathrm{i}U)(\mathbf{n})$. Therefore,
\[
\begin{aligned}
\tilde{E}(-\mathbf{n})&=-\frac{1}{2}((\eth^*)^2(Q+\mathrm{i}U)(-\mathbf{n})+\eth^2(Q-\mathrm{i}U)(-\mathbf{n}))\\
&=-\frac{1}{2}((\eth^*)^2(Q-\mathrm{i}U)(\mathbf{n})+\eth^2(Q+\mathrm{i}U)(\mathbf{n}))\\
&=\tilde{E}(\mathbf{n})
\end{aligned}
\]
and
\[
\begin{aligned}
\tilde{B}(-\mathbf{n})&=-\frac{1}{2\mathrm{i}}((\eth^*)^2(Q+\mathrm{i}U)(-\mathbf{n})-\eth^2(Q-\mathrm{i}U)(-\mathbf{n}))\\
&=-\frac{1}{2\mathrm{i}}((\eth^*)^2(Q-\mathrm{i}U)(\mathbf{n})-\eth^2(Q+\mathrm{i}U)(\mathbf{n}))\\
&=-\tilde{B}(\mathbf{n}).
\end{aligned}
\]
It means that $\tilde{E}(\mathbf{n})$ has even parity like electric field, while $\tilde{B}(\mathbf{n})$ has odd parity like magnetic field.

The spectral representation of the fields $\tilde{E}(\mathbf{n})$ and $\tilde{B}(\mathbf{n})$ has the form
\[
\begin{aligned}
\tilde{E}(\mathbf{n})&=\sum^{\infty}_{\ell=2}\sum^{\ell}_{m=-\ell}a_{\tilde{E},\ell m}Y_{\ell m}(\mathbf{n}),\\
\tilde{B}(\mathbf{n})&=\sum^{\infty}_{\ell=2}\sum^{\ell}_{m=-\ell}a_{\tilde{B},\ell m}Y_{\ell m}(\mathbf{n}),
\end{aligned}
\]
where
\[
\begin{aligned}
a_{\tilde{E},\ell m}&=-\frac{1}{2}\sqrt{\frac{(\ell+2)!}{(\ell-2)!}}(a_{2,\ell m}+a_{-2,\ell m}),\\
a_{\tilde{B},\ell m}&=-\frac{1}{2i}\sqrt{\frac{(\ell+2)!}{(\ell-2)!}}(a_{2,\ell m}-a_{-2,\ell m}).
\end{aligned}
\]

It is convenient to introduce the fields $E(\mathbf{n})$ and $B(\mathbf{n})$ as
\begin{equation}\label{eq:EB}
\begin{aligned}
E(\mathbf{n})&=\sum^{\infty}_{\ell=2}\sum^{\ell}_{m=-\ell}a_{E,\ell m}Y_{\ell m}(\mathbf{n}),\\
B(\mathbf{n})&=\sum^{\infty}_{\ell=2}\sum^{\ell}_{m=-\ell}a_{B,\ell m}Y_{\ell m}(\mathbf{n}),
\end{aligned}
\end{equation}
with
\begin{equation}\label{eq:coefficients}
\begin{aligned}
a_{E,\ell m}&=-\frac{1}{2}(a_{2,\ell m}+a_{-2,\ell m}),\\
a_{B,\ell m}&=-\frac{1}{2\mathrm{i}}(a_{2,\ell m}-a_{-2,\ell m}).
\end{aligned}
\end{equation}
The random fields $E(\mathbf{n})$ and $B(\mathbf{n})$ are scalar (spin $0$), real-valued, and isotropic. Moreover, $E(\mathbf{n})$ has even parity, while $B(\mathbf{n})$ has odd parity. The advantage of $E(\mathbf{n})$ and $B(\mathbf{n})$ is that their definition does not use assumption \eqref{eq:assumption}. The expansion coefficients $a_{E,\ell m}$ are called \emph{electric multipoles}, while the expansion coefficients $a_{B,\ell m}$ are called \emph{magnetic multipoles}.

Different notations for the fields $E(\mathbf{n})$ and $B(\mathbf{n})$ and electric and magnetic multipoles may be found in the literature. Some of them are shown in Table~\ref{tab3}. In what follows, we use notation by Lin and Wandelt (2006).

\begin{table}
\centering
\begin{tabular}{|l|c|c|}
\hline \textbf{Source} & \textbf{Fields} & \textbf{Multipoles} \\
\hline Challinor (2005), & & \\
Challinor and Peiris (2009) & --- & $E_{\ell m}$, $B_{\ell m}$ \\
\hline Durrer (2008), & $\mathcal{E}(\mathbf{n})$, $\mathcal{B}(\mathbf{n})$ & $e_{\ell m}$, $b_{\ell m}$ \\
\hline Geller and Marinucci (2008) & $f_{\mathbf{E}}$, $f_{\mathbf{M}}$ & $A_{\ell m\mathbf{E}}$, $A_{\ell m\mathbf{M}}$ \\
\hline Lin and Wandelt (2006), & $E(\mathbf{n})$, $B(\mathbf{n})$ & $a_{E,\ell m}$, $a_{B,\ell m}$ \\
\hline Weinberg (2008), & & \\
Zaldarriaga and Seljak (1997) & --- & $a_{E,\ell m}$, $a_{B,\ell m}$ \\
\hline
\end{tabular}
\caption{Examples of different notation for the fields $E(\mathbf{n})$ and $B(\mathbf{n})$  and its expansion coefficients.}\label{tab3}
\end{table}

We prove the following theorem.

\begin{theorem}\label{th:4}
Let $T(\mathbf{n})$ be a real-valued random field defined by \eqref{eq:temperature}. Let $(Q\pm\mathrm{i}U)(\mathbf{n})$ be random fields defined by \eqref{eq:polarisation} and \eqref{eq:conjugatepolarisation}. Let $E(\mathbf{n})$ and $B(\mathbf{n})$ be random fields \eqref{eq:EB} whose expansion coefficients are determined by \eqref{eq:coefficients}. The following statements are equivalent.

\begin{enumerate}

\item $((Q-\mathrm{i}U)(\mathbf{n}),T(\mathbf{n}),(Q+\mathrm{i}U)(\mathbf{n}))$ is an isotropic random field in $\xi_{-2}\oplus\xi_0\oplus\xi_2$. The fields $Q(\mathbf{n})$ and $U(\mathbf{n})$ are real-valued.

\item $(T(\mathbf{n}),E(\mathbf{n}),B(\mathbf{n}))$ is an isotropic random field in $\xi_0\oplus\xi_0\oplus\xi_0$ with real-valued components. The components $T(\mathbf{n})$ and $B(\mathbf{n})$ are uncorrelated. The components $E(\mathbf{n})$ and $B(\mathbf{n})$ are uncorrelated.

\end{enumerate}
\end{theorem}

\begin{proof}
Let $((Q-\mathrm{i}U)(\mathbf{n}),T(\mathbf{n}),(Q+\mathrm{i}U)(\mathbf{n}))$ be an isotropic random field in $\xi_{-2}\oplus\xi_0\oplus\xi_2$, and let $Q(\mathbf{n})$ and $U(\mathbf{n})$ be real-valued. By Theorem~\ref{th:3} and reality conditions \eqref{eq:ordinaryreality} and \eqref{eq:spinreality}, we have $\mathsf{E}[a_{T,\ell m}]=0$ for $\ell\neq 0$, $\mathsf{E}[a_{\pm 2,\ell m}]=0$ and
\begin{equation}\label{eq:polarisationisotropy}
\begin{aligned}
\mathsf{E}[a_{T,\ell m}\overline{a_{T,\ell'm'}}]&=\delta_{\ell\ell'}\delta_{mm'}C_{T,\ell},\\
\mathsf{E}[a_{\pm 2,\ell m}\overline{a_{\pm 2,\ell'm'}}]&=\delta_{\ell\ell'}\delta_{mm'}C_{2,\ell},\\
\mathsf{E}[a_{T,\ell m}\overline{a_{\pm 2,\ell'm'}}]&=\delta_{\ell\ell'}\delta_{mm'}C_{T,\pm 2,\ell},\\
\mathsf{E}[a_{-2,\ell m}\overline{a_{2,\ell'm'}}]&=\delta_{\ell\ell'}\delta_{mm'}C_{-2,2,\ell},
\end{aligned}
\end{equation}
with
\begin{equation}\label{eq:polarisationconvergence}
\sum^{\infty}_{\ell=0}(2\ell+1)C_{T,\ell}+2\sum^{\infty}_{\ell=2}(2\ell+1)C_{2,\ell}<\infty.
\end{equation}
Note that the second equation in \eqref{eq:polarisationisotropy} were proved for the general spin $s$ by Geller and Marinucci (2008) in their Theorem~7.2.

It is enough to prove that $\mathsf{E}[a_{E,\ell m}]=\mathsf{E}[a_{B,\ell m}]=0$ and
\begin{equation}\label{eq:isotropy}
\mathsf{E}[a_{X,\ell m}\overline{a_{Y,\ell'm'}}]=\delta_{\ell\ell'}\delta_{mm'}C_{XY,\ell}
\end{equation}
with
\begin{equation}\label{eq:convergence}
\sum_{\ell}(2\ell+1)C_{X,\ell}<\infty
\end{equation}
for all $X$, $Y\in\{T,E,B\}$. Then, the second statement of the theorem follows from Theorem~\ref{th:3}.

The first condition trivially follows from \eqref{eq:coefficients}. Condition \eqref{eq:isotropy} with $X=Y=T$ is obvious. We prove condition \eqref{eq:isotropy} with $X=Y=E$. Indeed, by \eqref{eq:coefficients} and \eqref{eq:spinreality},
\[
\begin{aligned}
\mathsf{E}[a_{E,\ell m}\overline{a_{E,\ell'm'}}]&=\frac{1}{4}(\mathsf{E}[(a_{2,\ell m}+a_{-2,\ell m})(\overline{a_{2,\ell'm'}}+\overline{a_{-2,\ell'm'}})])\\
&=\frac{1}{4}(\mathsf{E}[a_{2,\ell m}\overline{a_{2,\ell'm'}}]+\mathsf{E}[a_{2,\ell m}\overline{a_{-2,\ell'm'}}]\\
&\quad+\mathsf{E}[a_{-2,\ell m}\overline{a_{2,\ell'm'}}]+\mathsf{E}[a_{-2,\ell m}\overline{a_{-2,\ell'm'}}])\\
&=\frac{1}{2}\delta_{\ell\ell'}\delta_{mm'}(C_{2,\ell}+\RE C_{-2,2,\ell}).
\end{aligned}
\]
Condition \eqref{eq:isotropy} with $X=Y=B$ can be proved similarly.

Next, we prove condition \eqref{eq:isotropy} with $X=T$ and $Y=B$. Indeed,
\[
\begin{aligned}
\mathsf{E}[a_{T,\ell m}\overline{a_{B,\ell'm'}}]&=\frac{1}{2i}(\mathsf{E}[a_{T,\ell m}\overline{a_{2,\ell'm'}}]-\mathsf{E}[a_{T,\ell m}\overline{a_{-2,\ell'm'}}])\\
&=-\frac{1}{2}\delta_{\ell\ell'}\delta_{mm'}(C_{T,2,\ell}-C_{T,-2,\ell})\\
&=0
\end{aligned}
\]
by \eqref{eq:spinreality}, which also proves that $T(\mathbf{n})$ and $B(\mathbf{n})$ are uncorrelated. Condition~\eqref{eq:isotropy} for other cross-correlations can be proved similarly.

Next, we prove \eqref{eq:convergence} with $X=E$. Indeed,
\[
\sum_{\ell=2}^{\infty}(2\ell+1)C_{E,\ell}=\frac{1}{2}\sum_{\ell=2}^{\infty}(2\ell+1)
(C_{2,\ell}+\RE C_{-2,2,\ell})<\infty.
\]
Condition~\eqref{eq:convergence} for $X=B$ can be proved similarly.

Next, we prove that $E(\mathbf{n})$ is real-valued. It is enough to prove reality condition $a_{E,\ell\;-m}=(-1)^m\overline{a_{E,\ell m}}$. We have
\[
\begin{aligned}
a_{E,\ell\;-m}&=-\frac{1}{2}(a_{2,\ell\;-m}+a_{-2,\ell\;-m})\\
&=-\frac{1}{2}[(-1)^m\overline{a_{-2,\ell m}}+(-1)^{-m}\overline{a_{2,\ell m}}]\\
&=(-1)^m\overline{a_{E,\ell m}}.
\end{aligned}
\]
$B(\textbf{n})$ is real-valued by similar reasons.

Finally, we prove that $E(\mathbf{n})$ and $B(\mathbf{n})$ are uncorrelated. Indeed, $\mathsf{E}[E(\mathbf{n}_1)B(\mathbf{n}_2)]=C_{EB}(\mathbf{n}_1\boldsymbol{\cdot}\mathbf{n}_2)$, because $(T(\mathbf{n}),E(\mathbf{n}),B(\mathbf{n}))$ is an isotropic random field in $\xi_0\oplus\xi_0\oplus\xi_0$. So, $C_{EB}((-\mathbf{n}_1)\boldsymbol{\cdot}(-\mathbf{n}_2))=C_{EB}(\mathbf{n}_1\boldsymbol{\cdot}\mathbf{n}_2)$. On the other hand,
\[
\begin{aligned}
C_{EB}((-\mathbf{n}_1)\boldsymbol{\cdot}(-\mathbf{n}_2))&=\mathsf{E}[E(-\mathbf{n}_1)B(-\mathbf{n}_2)]\\
&=\mathsf{E}[E(\mathbf{n}_1)(-B(-\mathbf{n}_2))]\\
&=-C_{EB}(\mathbf{n}_1\boldsymbol{\cdot}\mathbf{n}_2),
\end{aligned}
\]
because $E(-\mathbf{n}_1)=E(\mathbf{n}_1)$ and $B(-\mathbf{n}_1)=-B(\mathbf{n}_1)$. Therefore, $C_{EB}(\mathbf{n}_1\boldsymbol{\cdot}\mathbf{n}_2)=0$.

Conversely, let $(T(\mathbf{n}),E(\mathbf{n}),B(\mathbf{n}))$ be an isotropic random field in $\xi_0\oplus\xi_0\oplus\xi_0$ with real-valued components, let the components $T(\mathbf{n})$ and $B(\mathbf{n})$ be uncorrelated, and let the components $E(\mathbf{n})$ and $B(\mathbf{n})$ be also uncorrelated. Solving system of equations \eqref{eq:coefficients}, we obtain
\begin{equation}\label{eq:inverse}
\begin{aligned}
a_{2,\ell m}&=-a_{E,\ell m}+a_{B,\ell m}\mathrm{i},\\
a_{-2,\ell m}&=-a_{E,\ell m}-a_{B,\ell m}\mathrm{i}.
\end{aligned}
\end{equation}
It is obvious that $\mathsf{E}[a_{\pm 2,\ell m}]=0$. We have to prove \eqref{eq:polarisationisotropy}, \eqref{eq:polarisationconvergence}, and \eqref{eq:spinreality}. The first equation in \eqref{eq:polarisationisotropy} is obvious. The second equation is proved as follows.
\[
\begin{aligned}
\mathsf{E}[a_{2,\ell m}\overline{a_{2,\ell'm'}}]&=\mathsf{E}[(-a_{E,\ell m}+a_{B,\ell m}\mathrm{i})(-a_{E,\ell'm'}-a_{B,\ell'm'}\mathrm{i})]\\
&=\delta_{\ell\ell'}\delta_{mm'}(C_{E,\ell}+C_{B,\ell}),
\end{aligned}
\]
because $E(\mathbf{n})$ and $B(\mathbf{n})$ are uncorrelated. Proof for negative coefficients is similar.

The third equation in \eqref{eq:polarisationisotropy} is proved as follows.
\[
\begin{aligned}
\mathsf{E}[a_{T,\ell m}\overline{a_{2,\ell'm'}}]&=\mathsf{E}[a_{T,\ell m}(-a_{E,\ell'm'}-a_{B,\ell'm'}\mathrm{i})]\\
&=-\delta_{\ell\ell'}\delta_{mm'}C_{TE,\ell},
\end{aligned}
\]
because $T(\mathbf{n})$ and $B(\mathbf{n})$ are uncorrelated. Proof for negative coefficient is similar.

The fourth equation in \eqref{eq:polarisationisotropy} is proved as follows.
\[
\begin{aligned}
\mathsf{E}[a_{-2,\ell m}\overline{a_{2,\ell'm'}}]&=\mathsf{E}[(-a_{E,\ell m}-a_{B,\ell m}\mathrm{i})(-a_{E,\ell'm'}-a_{B,\ell m}\mathrm{i})]\\
&=\delta_{\ell\ell'}\delta_{mm'}(C_{E,\ell}-C_{B,\ell}),
\end{aligned}
\]
because $E(\mathbf{n})$ and $B(\mathbf{n})$ are uncorrelated.

Because $C_{2,\ell}=C_{E,\ell}+C_{B,\ell}$, we have
\[
\begin{aligned}
\sum^{\infty}_{\ell=0}(2\ell+1)C_{T,\ell}+2\sum^{\infty}_{\ell=2}(2\ell+1)C_{2,\ell}
&=\sum^{\infty}_{\ell=0}(2\ell+1)C_{T,\ell}+2\sum^{\infty}_{\ell=2}(2\ell+1)(C_{E,\ell}+C_{B,\ell})\\
&<\infty
\end{aligned}
\]
which proves \eqref{eq:polarisationconvergence}. The reality condition \eqref{eq:spinreality} is proved as
\[
\begin{aligned}
\overline{a_{-2,\ell m}}&=\overline{-a_{E,\ell m}-a_{B,\ell m}\mathrm{i}}\\
&=-\overline{a_{E,\ell m}}+\overline{a_{B,\ell m}}\mathrm{i}\\
&=-(-1)^ma_{E,\ell\;-m}+(-1)^ma_{B,\ell\;-m}\mathrm{i}\\
&=(-1)^ma_{2,\ell\;-m}.
\end{aligned}
\]
\end{proof}

In the so called Gaussian cosmological theories, the random field $(T(\mathbf{n}),E(\mathbf{n}),B(\mathbf{n}))$ is supposed to be \emph{Gaussian} and isotropic with real-valued components. Let $\eta_{\ell 0j}$, $\ell\geq 0$, $1\leq j\leq 3$ and $\eta_{\ell mj}$, $\ell\geq 1$, $1\leq m\leq\ell$, $1\leq j\leq 6$ be independent standard normal random variables. Put
\[
\zeta_{\ell mj}=
\begin{cases}
\eta_{\ell 0j},&m=0,\\
\frac{1}{\sqrt{2}}(\eta_{\ell m\;2j-1}+\eta_{\ell m\;2j}\mathrm{i}),&m>0,
\end{cases}
\]
where $\ell\geq 0$, $0\leq m\leq\ell$, and $1\leq j\leq 3$. Now put
\[
\begin{aligned}
a_{T,\ell m}&=(C_{T,\ell})^{1/2}\zeta_{\ell m1},\\
a_{E,\ell m}&=\frac{C_{TE,\ell}}{(C_{T,\ell})^{1/2}}\zeta_{\ell m1}+\left(C_{E,\ell}-\frac{(C_{TE,\ell})^2}{C_{T,\ell}}\right)^{1/2}\zeta_{\ell m2},\\
a_{B,\ell m}&=(C_{B,\ell m})^{1/2}\zeta_{\ell m3},
\end{aligned}
\]
for $m\geq 0$ and $a_{X,\ell\;-m}=(-1)^m\overline{a_{X,\ell m}}$ for $m<0$ and $X\in\{T,E,B\}$. The random fields
\[
\begin{aligned}
T(\mathbf{n})&=\sum^{\infty}_{\ell=0}\sum^{\ell}_{m=-\ell}a_{T,\ell m}Y_{\ell m}(\mathbf{n}),\\
E(\mathbf{n})&=\sum^{\infty}_{\ell=2}\sum^{\ell}_{m=-\ell}a_{E,\ell m}Y_{\ell m}(\mathbf{n}),\\
B(\mathbf{n})&=\sum^{\infty}_{\ell=2}\sum^{\ell}_{m=-\ell}a_{B,\ell m}Y_{\ell m}(\mathbf{n})
\end{aligned}
\]
satisfy all conditions of the second statement of Theorem~\ref{th:4}. The random fields $(Q\pm\mathrm{i}U)(\mathbf{n})$ can be reconstructed by \eqref{eq:inverse}, \eqref{eq:polarisation}, and \eqref{eq:conjugatepolarisation}. By Theorem~\ref{th:4}, $((Q-\mathrm{i}U)(\mathbf{n}),T(\mathbf{n}),(Q+\mathrm{i}U)(\mathbf{n}))$ is an isotropic Gaussian random field in $\xi_{-2}\oplus\xi_0\oplus\xi_2$. The fields $Q(\mathbf{n})$ and $U(\mathbf{n})$ are real-valued.

Finally, we note that Kamionkowski et al (1997) proposed a different formalism for computations of the polarisation field on the whole sky. Instead of spin-weighted harmonics ${}_sY_{\ell m}$, they use \emph{tensor harmonics} $Y^E_{\ell m}$ and $Y^B_{\ell m}$ which are related to the spin-weighted harmonics as follows.
\[
\begin{aligned}
Y^E_{\ell m}&=\frac{1}{\sqrt{2}}({}_{-2}Y_{\ell m}\mathbf{e}_-\otimes\mathbf{e}_-+{}_2Y_{\ell m}\mathbf{e}_+\otimes\mathbf{e}_+),\\
Y^B_{\ell m}&=\frac{1}{\mathrm{i}\sqrt{2}}({}_{-2}Y_{\ell m}\mathbf{e}_-\otimes\mathbf{e}_--{}_2Y_{\ell m}\mathbf{e}_+\otimes\mathbf{e}_+).
\end{aligned}
\]
This formalism is also used by Cabella and Kamionkowski (2005), Challinor (2004, 2009), Challinor and Peiris (2009) among others. An excellent survey of different types of spherical harmonics may be found in Thorne (1980).


\begin{thebibliography}{99}

\bibitem{B} A.~O.~Barut and R.~R\k{a}czka. \emph{Theory of group representations and applications}. Second edition. World Scientific, Singapore, 1986.

\bibitem{C3} P.~Cabella and M.~Kamionkowski. Theory of cosmic microwave background polarization. \href{http://arxiv.org/PS_cache/astro-ph/pdf/0403/0403392v2.pdf}{arXiv:astro-ph/0403392v2 18 Mar 2005}.

\bibitem{C} R.~Camporesi. The Helgason Fourier transform for homogeneous vector bundles over compact Riemannian symmetric spaces~--- the local theory. \emph{J. Funct. Anal.} \textbf{220} (2005), 97--117.

\bibitem{C5} A.~Challinor. Anisotropies in the cosmic microwave background. \href{http://arxiv.org/PS_cache/astro-ph/pdf/0403/0403344v1.pdf}{arXiv:astro-ph/0403344v1 15 Mar 2004}.

\bibitem{C4} A.~Challinor. Cosmic microwave background polarisation analysis. In V.~J.~Martinez, E.~Saar, E.~Mart\'{\i}nez-Gonz\'{a}lez and M.-J.~Pons-Borderia (eds.) \emph{Data analysis in cosmology} (\emph{Lect. Notes Phys.}, vol.~665), Springer, Berlin, 2009.

\bibitem{C2} A.~Challinor and H.~Peiris. Lecture notes on the physics of cosmic microwave background anisotropies. \href{http://arxiv.org/PS_cache/arxiv/pdf/0903/0903.5158v1.pdf}{arXiv:0903.5158v1 [astro-ph.CO] 30 Mar 2009}.

\bibitem{D} R.~Durrer. \emph{The cosmic microwave background}. Cambridge University Press, Cambridge, 2008.

\bibitem{GS} I.~M.~Gelfand and Z.~Ya.~Shapiro. Representations of the group of rotations in three-dimensional space and their applications. \emph{Uspehi Matem. Nauk (N.S.)} \textbf{7}, no.~1(47) (1952), 3--117 (Russian).

\bibitem{G2} D.~Geller, X.~Lan, and D.~Marinucci. Spin needlets spectral estimation. \href{http://arxiv.org/PS_cache/arxiv/pdf/0907/0907.3369v1.pdf}{arXiv:0907.3369v1 [math.ST] 20 Jul 2009}.

\bibitem{G} D.~Geller and  D.~Marinucci. Spin wavelets on the sphere. \href{http://arxiv.org/PS_cache/arxiv/pdf/0811/0811.2935v3.pdf}{arXiv:0811.2935v3 [math.CA] 15 Dec 2008}.

\bibitem{L} Y.-T.~Lin and B.~D.~Wandelt. A beginner's guide to the theory of CMB temperature and polarization power spectra in the line-of-sight formalism. \emph{Astroparticle Physics} \textbf{25} (2006), 151--166.

\bibitem{K} M.~Kamionkowski, A.~Kosowsky and A.~Stebbins. Statistics of cosmic microwave background polarization. \emph{Phys. Rev. D} \textbf{55} (1997), no.~12, 7368--7388.

\bibitem{N} E.~T.~Newman and R.~Penrose. Note on the Bondi--Metzner--Sachs Group. \emph{J. Math. Phys.} \textbf{7}, no.~5 (1966), 863--870.

\bibitem{O} A.~M.~Obukhov. Statistically homogeneous random fields on a sphere. \emph{Uspehi Mat. Nauk} \textbf{2} (1947), no.~2, 196--198.

\bibitem{R} Yu.~A.~Rozanov. Spectral theory of $n$-dimensional stationary stochastic processes with discrete time. \emph{Uspehi Mat. Nauk} \textbf{13} (1958), no.~2(80), 93--142 (Russian).

\bibitem{T} K.~S.~Thorne. Multipole expansions of gravitational radiation. \emph{Rev. Mod. Phys.} \textbf{52} (1980), no.~2, 299--339.

\bibitem{V} N.~Ya.~Vilenkin. \emph{Special functions and the theory of group representations}. \emph{Translations of Mathematical Monographs}, \textbf{22}. American Mathematical Society, Providence, R.I., 1968.

\bibitem{W} S.~Weinberg. \emph{Cosmology}. Oxford University Press, Oxford, 2008.

\bibitem{Y} A.~M.~Yaglom. Second-order homogeneous random fields. In \emph{Proc. 4th Berkeley Sympos. Math. Statist. and Prob.}, vol. II, pp. 593--622, Berkeley, CA, 1961. Univ. California Press.

\bibitem{Z} M.~Zaldarriaga and U.~Seljak. An all-sky analysis of polarisation in the microwave background. \emph{Phys. Rev. D} \textbf{55} (1997), no.~4, 1830--1840.

\end{thebibliography}
\end{document}